\documentclass{article}
\usepackage[english]{babel}
\usepackage{amsthm}
\usepackage[square, numbers]{natbib}
\newcommand{\EndProof}{}

\newtheorem{definition}{Definition}[section]
\newtheorem{lemma}{Lemma}[section]
\newtheorem{example}{Example}[section]
\newtheorem{remark}{Remark}[section]
\newtheorem*{claim}{Claim}
\newtheorem{theorem}{Theorem}[section]
\newtheorem{corollary}{Corollary}[section]
\newtheorem{question}{Question}

\newcommand{\Fil}{\mathcal{F}}
\newcommand{\Ul}{\mathcal{U}}
\newcommand{\Nat}{\mathbb{N}}
\newcommand{\zwakkePfilter}{weak P$^+$-filter}
\newcommand{\sterkePfilter}{P$^+$-filter}
\newcommand{\Peigenschap}{P$^+$-property}

\usepackage{amsmath, amssymb, amsfonts,graphicx,url}

\begin{document}

\title{{Filter-dependent versions} of the\\ Uniform Boundedness Principle}
	
	\author{Ben De Bondt\footnote{Ben De Bondt gratefully acknowledges support by Ghent University, through BOF PhD grant BOF18/DOC/298, which supported both the research and the preparation of the current article.} \footnote{Currently at Universit\'e de Paris, Institut de Math\'ematiques de Jussieu-Paris Rive Gauche, as a Cofund MathInParis PhD fellow. The Cofund MathInParis PhD project has received funding from the European Union’s Horizon 2020 research and innovation programme under the Marie Sk\l{}odowska-Curie grant agreement No. 754362.
	\protect \includegraphics[width=0.45truecm]{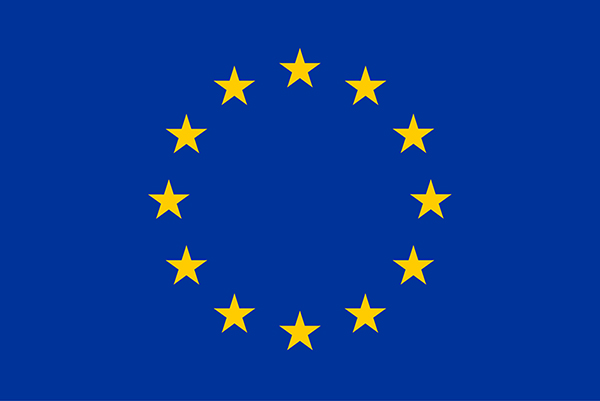}}
			\\
		\phantom{Ghent University}  \\
		\and 
		Hans Vernaeve\footnote{Hans Vernaeve gratefully acknowledges support by research Grant 1.5.129.18N of the Research Foundation Flanders FWO.} \\
		\phantom{Ghent University} \\
	}
	\date{}
	\maketitle
	
	\begin{abstract}
	\noindent
		For every filter~$\Fil$ on $\Nat,$  we introduce and study corresponding {\it uniform\linebreak
        $\Fil$-boundedness principles for locally convex topological vector spaces}. These principles generalise the classical uniform boundedness principles for sequences of continuous linear maps by coinciding with these principles when the filter~$\Fil$ equals the Fr\'echet filter of cofinite subsets of $\Nat.$ We determine combinatorial properties for the filter~$\Fil$ which ensure that these uniform $\Fil$-boundedness principles hold for every Fr\'echet space.
        Furthermore, for several types of Fr\'echet spaces, we also isolate properties of $\Fil$ that are necessary for the validity of {these} uniform $\Fil$-boundedness principles.
        For every infinite-dimensional Banach space~$X$, we obtain in this way exact combinatorial characterisations of those filters~$\Fil$ for which the corresponding uniform $\Fil$-boundedness principles hold true for~$X$.
	\end{abstract}
	
	Banach spaces

\newpage
\section{Introduction}
This paper contributes to the line of research (\cite{kadets}, \cite{kadets3}, \cite{connor}, \cite{filipow}, \cite{ganichev}, \cite{kadets2},\linebreak \cite{kadets4},   \ldots)
that is concerned with the study of {\it filter versions}\footnote{
    Because of the duality between filters and ideals, one can equivalently study the cor\-res\-pon\-ding {\it ideal versions} of such theorems. Several authors choose to do so.}
    of several theorems in analysis {(or infinite combinatorics)}, whose classical proofs critically rely on Baire-categoric, measure-theoretic, or Ramsey-theoretic ideas. Before clarifying how to derive in a canonical way the statement $T_\Fil$ for a filter version of the theorem $T$, let us quickly motivate the interest in the new statement $T_\Fil$ (see also the introduction of \cite{kadets4}). As well as in generalising the theorem $T$ involved, additional motivation for the study of the filter version $T_\Fil$ lies in gaining understanding of the combinatorial principles that are needed to prove $T$. Indeed, those filters on $\Nat$ that preserve validity of the filter versions often turn out to possess  interesting combinatorial properties, closely related to the \mbox{P-prop}\-er\-ties and \mbox{Q-prop}\-er\-ties, that are classically studied in set theory.
 In the present paper, we show that this is true as well for filter versions of the uniform boundedness principle for Fr\'echet spaces.
 \smallskip
 
 \noindent
 One main corollary of the uniform boundedness principle entails continuity of the pointwise limits of sequences of continuous linear mappings.
 In \cite{kolk} it was shown that one particular filter version of this consequence of the UBP holds for separable Banach spaces whenever the dual ideal has the property denoted in \cite{kolk} by property (APO).\smallskip

\noindent Another important consequence of the UBP (weak convergence implies boundedness in Banach spaces) was studied from the filter perspective in \cite{connor}, \cite{ganichev} (see Remark \ref{nieuwe opm} below).\smallskip

 \noindent
 In \cite{kadets3}, Avil{\'e}s, Kadets, P{\'e}rez and Solecki studied the
{\it Baire filters}\footnote{Or rather, the ideals dual to such filters.}, which have the defining property that for every complete metric space~$X$ and every order-reversing map $$f: (\Fil, \subseteq) \to (\{F \subseteq X: F \text{ nowhere dense in } X \}, \subseteq),$$ one has that $$\bigcup_{A\in \Fil} f(A) \neq X.$$
 Given that the uniform boundedness principle comes as a consequence of the Baire category theorem, it is not surprising that the uniform \mbox{$\Fil$-bound}\-ed\-ness principle will hold for every Fr\'echet space whenever $\Fil$ is a (free) Baire filter (see  Lemma~\ref{Bairefilter} below).
 However, it was found in \cite{kadets3}, that the notion of Baire filter is a rather restrictive one. The analytic Baire filters for instance, are all countably generated. This makes it natural to determine weaker conditions under which filter versions of the uniform boundedness principle are true.\footnote{In particular,  Example \ref{vb} will exhibit an analytic filter of uncountable character for which the uniform \mbox{$\Fil$-bound}\-ed\-ness principle holds for every Fr\'echet space.}

\subsection{Terminology and notation}
We now proceed to the introduction of the key concepts under consideration.\\
Every filter~$\mathcal{F}$ on $\mathbb{N}$ gives rise to the generalised quantifier $(\forall_{\!\mathcal{F}}\, i \in \mathbb{N})$ defined by the following rule:
$$(\forall_{\!\mathcal{F}}\, i \in \mathbb{N})(\Phi(i)) \quad \Leftrightarrow \quad \{ i\in \mathbb{N}: \Phi(i) \} \in \mathcal{F},$$
for every formula $\Phi.$
We let $\mathcal{C}$ be the Fr\'echet filter consisting of the cofinite subsets of $\mathbb{N}$ and let $\mathcal{F}$ be an arbitrary filter on $\mathbb{N}.$ Then, any concept which can be expressed in terms of the quantifier  $(\forall_{\!\mathcal{C}}\, i \in \mathbb{N})$ has a direct \mbox{$\mathcal{F}$-analogue}, obtained by replacing the quantifier  $(\forall_{\!\mathcal{C}}\, i \in \mathbb{N})$ by the quantifier  $(\forall_{\!\mathcal{F}}\, i \in \mathbb{N}).$
For example, convergence of a sequence $(x_i)_i$ in a metric space $(M,d)$ to $x \in M$ can be expressed by the formula
$$(\forall \varepsilon \in \mathbb{R}_{>0})(\forall_{\!\mathcal{C}}\, i \in \mathbb{N}) \quad d(x,x_i) < \varepsilon,$$
and the replacement of generalised quantifiers gives the useful concept of \mbox{$\Fil$-con}\-ver\-gence of the sequence $(x_i)_i$ to $x$:
 $$(\forall \varepsilon \in \mathbb{R}_{>0})(\forall_{\!\mathcal{F}}\, i \in \mathbb{N}) \quad d(x,x_i) < \varepsilon.$$
 In the same fashion, we now define \mbox{$\Fil$-analogues} for two principal concepts from the theory of topological vector spaces, that of boundedness of a sequence and that of equicontinuity of a sequence of (continuous linear) maps. In what follows, all sequences are indexed by the non-negative integers, unless stated otherwise.\\
 Let $X$ be a topological vector space and $\Fil$ a filter on $\Nat$. A sequence $(x_i)_i$ in $X$ is \mbox{\bf $\mathcal{F}$-bounded} if for every \mbox{$0$-neigh}\-bour\-hood $U$ in $X$ there exists $t>0$ such that
$$ tx_i \in U \qquad (\forall_{\!\mathcal{F}}\, i \in \mathbb{N}).$$
Note that in a normed space~$X$, $(x_i)_i$ is \mbox{$\mathcal{F}$-bounded} {if and only if} there exists $C>0$ such that
$$ \|x_i\|\le C\qquad (\forall_{\!\mathcal{F}}\, i \in \mathbb{N}).$$
Also note that for general filters $\mathcal{F}$, \mbox{$\Fil$-bound}\-ed\-ness is not necessarily preserved by rearrangements of sequences.\\
Given two topological vector spaces $X,Y$, let $\mathcal{L}(X,Y)$ denote the vector space of continuous linear maps from $X$ to $Y$.\\
A sequence $(T_i)_i$ in $\mathcal{L}(X,Y)$ is {\bf \mbox{$\Fil$-equi}\-con\-tin\-uous} if for every \mbox{$0$-neigh}\-bour\-hood $V$ in $Y$ there exists a \mbox{$0$-neigh}\-bour\-hood $U$ in $X$ such that
$$ T_i[U] \subseteq V \qquad (\forall_{\!\mathcal{F}}\, i \in \mathbb{N}).$$
The reader can observe that if $X$ and $Y$ are both normed spaces, a sequence $(T_i)_i$ of continuous linear maps from $X$ to $Y$ is \mbox{$\Fil$-equi}\-con\-tin\-uous precisely when it is \mbox{$\Fil$-bounded} in the normed space $\mathcal{L}(X,Y)$ of continuous linear operators from $X$ to $Y.$\\ Continuing in this same vein, a sequence $(T_i)_i$ in $\mathcal{L}(X,Y)$ is defined to be {\bf pointwise \mbox{$\mathcal{F}$-bounded}} when for every $x \in X,$ the sequence $(T_i(x))_i$ is an \mbox{$\mathcal{F}$-bounded} sequence in~$Y.$ Equivalently, $(T_i)_i$ is pointwise \mbox{$\mathcal{F}$-bounded} when it is \mbox{$\Fil$-bounded} in the space $\mathcal{L}(X,Y)$ equipped with the topology of pointwise convergence.\\
Given the topological vector spaces $X,Y$ and the filter~$\Fil$ on $\Nat,$ let $PW_{\Fil}(X,Y)$ and $EQ_{\Fil}(X,Y)$ denote respectively the set of all pointwise \mbox{$\mathcal{F}$-bounded} sequences in $\mathcal{L}(X,Y)$ and the set of all \mbox{$\Fil$-equi}\-con\-tin\-uous sequences in $\mathcal{L}(X,Y)$.\\ It is clear that $EQ_{\Fil}(X,Y) \subseteq PW_{\Fil}(X,Y).$ The validity of the reverse inclusion is the subject of uniform \mbox{$\Fil$-bound}\-ed\-ness principles. We write $UBP_{\Fil}(X,Y)$ as abbreviation for the statement
$PW_{\Fil}(X,Y) = EQ_{\Fil}(X,Y)$ and say that {\it the uniform \mbox{$\Fil$-bound}\-ed\-ness principle holds for continuous linear maps from $X$ to $Y$} whenever this statement is true.
We will also say that the uniform \mbox{$\Fil$-bound}\-ed\-ness principle holds for $X$ if $UBP_\Fil(X,Y)$ holds for every locally convex space~$Y.$
We define $\mathcal{F}$ to be a {\bf Fr\'echet-UBP-filter} if $UBP_{\Fil}(X,Y)$ holds for every Fr\'echet space~$X$ and every locally convex space~$Y.$ By the classical theorem of Banach and Steinhaus, the Fr\'echet filter~$\mathcal{C}$ is a Fr\'echet-UBP-filter.\\We define $\mathcal{F}$ to be a {\bf Banach-UBP-filter} if $UBP_{\Fil}(X,Y)$ holds for every Banach space~$X$ and every locally convex space~$Y.$ \bigskip$\;$\bigskip

\noindent
The key concepts being introduced above, we can now state our main objectives.\\ Before doing so, we briefly illustrate the concepts just introduced by checking that the Baire filters studied in \cite{kadets3} are indeed Fr\'echet-UBP-filters.

\begin{lemma}$\;$\label{Bairefilter}\\
    If the filter $\Fil$ (on $\Nat$) is a Baire filter containing every cofinite subset of $\Nat$, then $UBP_{\Fil}(X,Y)$ holds for every Fr\'echet space $X$ and every topological vector space~$Y.$
\end{lemma}

\begin{proof}
Let $(T_i)_i$ be a pointwise \mbox{$\Fil$-bounded} sequence in $\mathcal{L}(X,Y)$. We mimic the Baire category proof of the uniform boundedness principle to prove that $(T_i)_i$ is an \mbox{$\Fil$-equi}\-con\-tin\-uous sequence. Let $W$ be an arbitrary $0$-neighbourhood in $Y$ and choose a second $0$-neigh\-bour\-hood $V$ in $Y$ which is closed and balanced, and satisfies $V-V \subseteq W$.\\
For every $F\in\mathcal{F}$ and $n\in\Nat$, we define a corresponding set $C_{F,n} \subseteq X$ as follows:  $$C_{F,n}:=\bigcap_{i\in F}T_i^{-1}[nV].$$
Then, the pointwise \mbox{$\Fil$-bound}\-ed\-ness of $(T_i)_i$ allows us to write:
$$X=\bigcup_{F\in \Fil}\bigcup_{n\in \Nat}C_{F,n}=\bigcup_{F\in\Fil}C_{F,\min(F)}.$$
Since $f:(\Fil,\subseteq)\to (P(X), \subseteq):F\mapsto C_{F,\min(F)}$ is an order-reversing map with $C_{F,\min(F)}$ closed for any $F\in\mathcal{F}$, it follows from the defining property of Baire filters that for some $F\in\Fil$, the set $C_{F,\min(F)}$ has non-empty interior.\\
Therefore, there exist $x\in X$, a $0$-neighbourhood $U$ in $X$ and $n \in \Nat$ such that
$$ T_i[x+U] \subseteq nV \qquad (\forall_{\!\mathcal{F}}\, i \in \mathbb{N}).$$
Hence,
$$ T_i\!\left[\frac{1}{n}\,U\right] \subseteq V-V \subseteq W \qquad (\forall_{\!\mathcal{F}}\, i \in \mathbb{N}).$$
\EndProof
\end{proof}
    
\subsection{Aims}
We will answer the following questions.
\begin{enumerate}
    \item Is every Banach-UBP-filter a Fr\'echet-UBP-filter? (Theorem~\ref{mainx}.)
\item Does there exist a combinatorial characterisation of the Fr\'echet- or Banach-UBP-filters? (Theorem~\ref{mainx}.)
\item Do there exist Fr\'echet-UBP-filters that are not countably generated?\\ (Example~\ref{vb}.)
\item If $X$ is an infinite-dimensional Banach space, does the validity of the uniform \mbox{$\Fil$-bound}\-ed\-ness principle for $X$ depend on the relationship between the space~$X$ and the filter~$\mathcal{F}$, or solely on the filter~$\Fil$? (Theorem~\ref{mainx}.)
\item Is the existence of either Fr\'echet- or Banach-UBP-ultrafilters consistent with $\mathsf{ZFC}?$ Is the existence of such ultrafilters provable in $\mathsf{ZFC}$?\\ (Corollary~\ref{cor-ZFC}.)
\end{enumerate}

\section{Combinatorics of filters on a countable set}
A filter on a set $S$ is a non-empty $\subseteq$-upwards closed set $\mathcal{F} \subsetneq P(S)$ (where $P(S)$ denotes the power set of $S$) which is closed under finite intersections.
Throughout this paper, we will always have $S = \Nat := \{0,1,2,3, \ldots\}$ and will only consider filters which are free, i.e.\ we will make the extra assumption that $\mathcal{C} \subseteq \mathcal{F}.$ The corresponding questions for non-free filters can be reduced to the case of free filters.\\
An ideal on a set $S$ is a non-empty $\subseteq$-downwards closed set $\mathcal{I} \subsetneq P(S)$ which is closed under finite unions.
If $\Fil$ is a filter on $S$, the set $\Fil^\ast= \{A \subseteq S : A^c \in \Fil \}$ is the dual ideal of $\Fil.$ We will denote the set $P(S)\setminus \Fil^\ast$ of \mbox{$\Fil$-sta}\-tion\-ary sets by $\Fil^+.$ Note that a subset of $S$ is \mbox{$\Fil$-sta}\-tion\-ary precisely when it has non-empty intersection with every element of~$\Fil.$\\
For two sets $A,B \subseteq \Nat,$ we write $A\subseteq^\ast B$ whenever $A$ is contained in $B$ modulo a finite set, i.e.\ $A \setminus B$ is finite. For $A\subseteq \Nat$ infinite, the enumerating function of $A$ is the unique strictly increasing surjection $\eta_A: \Nat \to A.$\\
We will show that the uniform boundedness principle for filters is closely related to the following combinatorial properties for filters on $\Nat.$

\begin{definition}$\;$\\ Let $\Fil$ be a filter on $\Nat,$ then $\Fil$ is
\begin{itemize}
\item  a {\bf \zwakkePfilter} if for every decreasing sequence $(A_k)_k$ of sets in $\mathcal{F}$ and for every \mbox{$\Fil$-sta}\-tion\-ary set $I\subseteq \Nat$ there exists an \mbox{$\Fil$-sta}\-tion\-ary set $B \subseteq I$ such that $B \subseteq^\ast A_k$
for every $k.$
\item a {\bf \sterkePfilter} if for every decreasing sequence $(A_k)_k$ of \mbox{$\Fil$-sta}\-tion\-ary sets, there exists an \mbox{$\Fil$-sta}\-tion\-ary set $B$ such that $B \subseteq^\ast A_k$
for every $k.$
\item a {\bf rapid$^+$ filter} if for every \mbox{$\Fil$-sta}\-tion\-ary set $I\subseteq \Nat$ and every strictly increasing function $f: \Nat \to \Nat,$ there exists an \mbox{$\Fil$-sta}\-tion\-ary set $B \subseteq I$ such that the function enumerating $B$ dominates $f,$ i.e.\ $f\leq \eta_B$.
\end{itemize}
\end{definition}

\noindent
The rapid$^+$- and (weak) $P^+$-properties of filters, along with several closely related variations, have appeared in prior literature (see \cite{pino}, \cite{hrusak}, \cite{cardinalinvariants}, \ldots).
We advise the reader to be careful while consulting this literature as the terminology used to denote combinatorial properties of filters does vary among authors.
Let us point out that the study of these and related properties of filters has originated in the study of ultrafilters on $\Nat$. For free ultrafilters the weak {\Peigenschap}, the {\Peigenschap} and several other related properties coincide and they determine a topological invariant of points in the Stone-\v{C}ech remainder $\beta \Nat \setminus \Nat$. In \cite{rudin},
this $P$-property was used by W.\ Rudin to study non-homogeneity of the space $\beta \Nat \setminus \Nat$.\\
In the following Lemma~\ref{alternatievekaraktrerisatieP+}, Lemma~\ref{alternatievekaraktrerisatiePx} and Lemma~\ref{alternatievekaraktrerisatierapid}, we formulate well-known alternative characterisations of respectively {\zwakkePfilter s}, {\sterkePfilter s} and rapid$^+$ filters, which will be of value further on. The proof of Lemma~\ref{alternatievekaraktrerisatiePx} is a straightforward modification of the proof of Lemma~\ref{alternatievekaraktrerisatieP+}, so we only give the latter. See also \cite{pino} and \cite{cardinalinvariants} for similar proofs of these (and closely related) results.

\begin{lemma}$\;$\label{alternatievekaraktrerisatieP+}\\
A filter $\Fil$ on $\Nat$ is a {\zwakkePfilter} if and only if for every function $\Gamma: \Nat \to \Nat$ one of the following two statements holds:
\begin{itemize}
\item $\Gamma$ is bounded on some \mbox{$\Fil$-sta}\-tion\-ary set,
\item every \mbox{$\Fil$-sta}\-tion\-ary set $I$ has an \mbox{$\Fil$-sta}\-tion\-ary subset on which $\Gamma$ is finite-to-one.
\end{itemize}
\end{lemma}

\begin{proof}
Suppose first that $\Fil$ is a {\zwakkePfilter} and that $\Gamma: \Nat \to \Nat$ is unbounded on every \mbox{$\Fil$-sta}\-tion\-ary set. Then, for every $k$, the set $A_k:=\{
n\in \Nat: \Gamma(n)>k
\}$ belongs to~$\Fil$.
This implies that for every \mbox{$\Fil$-sta}\-tion\-ary set $I$, there exists an \mbox{$\Fil$-sta}\-tion\-ary set $B\subseteq I$ such that $B\subseteq^{\ast} A_k$ for every $k$. By consequence, $\Gamma$ is finite-to-one on~$B$.\\
Suppose next that $\Fil$ has the property of Lemma~\ref{alternatievekaraktrerisatieP+}. Let $(A_k)_k$ be a decreasing sequence of sets in $\Fil$. Let $A'_k=A_k\cap [k,+\infty[$. Define $\Gamma:\Nat\to\Nat$ as follows:
$$\begin{cases}
\Gamma(n)=\max\{ l: n\in A'_l \}&\text{if } n\in A'_0,\\
\Gamma(n)=0&\text{if }n\notin A'_0.
\end{cases}$$
Now, $\{ n\in \Nat: \Gamma(n)\leq k  \}  \subseteq \Nat \setminus A'_{k+1}$ cannot be \mbox{$\Fil$-sta}\-tion\-ary. Hence, for every \mbox{$\Fil$-sta}\-tion\-ary set $I$, there is an \mbox{$\Fil$-sta}\-tion\-ary set $B\subseteq I$ on which $\Gamma$ is finite-to-one. This implies that $B\subseteq^{\ast} A'_k\subseteq A_k$ for every $k$.         
    \EndProof
\end{proof}

\begin{lemma}$\;$\label{alternatievekaraktrerisatiePx}\\
A filter $\Fil$ on $\Nat$ is a {\sterkePfilter} if and only if for every function $\Gamma: \Nat \to \Nat$ one of the following two statements holds:
\begin{itemize}
\item $\Gamma$ is bounded on some element of $\Fil$,
\item $\Gamma$ is finite-to-one on some \mbox{$\Fil$-sta}\-tion\-ary set.
\end{itemize}
\end{lemma}
\bigskip

\noindent
If $I \subseteq \Nat,$ we say that $I' \subseteq I$ is an interval of $I$ whenever $I' = I \cap [a,b)$ for $a,b \in \Nat.$

\begin{definition}$\;$\\
An increasing interval-partition
     of $I$ is a partition of $I$ into intervals $(I_n)_n$ of $I$ such that $\max(I_n) < \min(I_m)$ for each $n<m.$
\end{definition}

\noindent
To a strictly increasing function $f: \mathbb{N} \to I,$ we associate the increasing interval-partition $(I^f_n)_n$ defined by $I^f_n = I \cap (f(n-1), f(n)]$ (where we agree on\linebreak $f(-1)= -1$ for all such~$f$). Clearly, every increasing interval-partition is of the form $(I^f_n)_n$ for a unique strictly increasing function $f:\Nat \to I.$

\begin{definition}$\;$\\
Let $\mathcal{F}$ be a filter on $\mathbb{N}$ and $I\subseteq \Nat$ \mbox{$\Fil$-sta}\-tion\-ary.
Let $(I_n)_n$ be an increasing interval-partition of $I.$ We say that $\mathcal{F}$ is slow with respect to $(I_n)_n$ if every \mbox{$\Fil$-sta}\-tion\-ary $A \subseteq I$ satisfies
$(\exists n \in \mathbb{N})(\, |A \cap I_n| > n).$
\end{definition}

\begin{lemma}$\;$\label{alternatievekaraktrerisatierapid}\\
If $\mathcal{F}$ is a filter on $\mathbb{N}$, then the following two statements are equivalent:
\begin{enumerate}
\item $\Fil$ is rapid$^+$.
\item No \mbox{$\Fil$-sta}\-tion\-ary set $I \subseteq \mathbb{N}$ has an increasing interval-partition $(I_n)_n$ such that $\Fil$ is slow with respect to $(I_n)_n.$
\end{enumerate}
\end{lemma}

    \begin{proof}$\;$\\
The implication $1. \Rightarrow 2.$ is easily verified, we prove the implication $2. \Rightarrow 1.$\\
        Let $I\subseteq \mathbb{N}$ be \mbox{$\Fil$-sta}\-tion\-ary. Let $f:\mathbb{N}\to I$ be strictly increasing. We prove that there exists a $B\subseteq I$ \mbox{$\Fil$-sta}\-tion\-ary for which $f\leq \eta_B$, given that, for each increasing interval-partition $(I_n)_n$ of $I$, $\mathcal{F}$ is not slow with respect to $(I_n)_n$.\\        Define $\displaystyle q:\mathbb{N}\to\mathbb{N}:q(n)=\frac{n(n+1)}{2}$.\\        Consider the following increasing interval-partition $(I_n)_n$ of $I$:        $$ I_0=[0,f(q(1))]\cap I, \quad I_n=]f(q(n)),f(q(n+1))]\cap I \quad \forall n>0.$$    By assumption $\mathcal{F}$ is not slow with respect to $(I_n)_n$, so we can choose $B\subseteq I$ \mbox{$\Fil$-sta}\-tion\-ary, for which $$(\forall n\in \mathbb{N})\ |B\cap I_n|\leq n.$$It now suffices to show that for such $B\subseteq I$ \mbox{$\Fil$-sta}\-tion\-ary we also have        $$ (\forall n\in \mathbb{N})\  |B\cap [0,f(n)]|\leq n.$$        We can see this as follows. Given $n$, choose $r\in\mathbb{N}$ maximal such that $q(r)\leq n$, then $f(n) < f(q(r+1))$, so $\displaystyle B\cap [0,f(n)] \subseteq \bigcup_{l\leq r}B\cap I_{l},$ but then $$ |B\cap [0,f(n)]| \leq |\bigcup_{l\leq r}B\cap I_{l}| \leq \sum_{l\leq r} l =q(r)\leq n.$$    \EndProof
        \end{proof}
    
\noindent
Every countably generated filter~$\Fil$ on $\mathbb{N}$ is both a rapid$^+$ filter and a \sterkePfilter. We now illustrate these properties with three important examples of filters. The first two examples originate in real analysis.

\begin{example}$\;$\\
The asymptotic density filter~$\Fil_d$ consists of those sets $A\subseteq \Nat$ that satisfy
$$ \lim_{n\to \infty} \frac{ |\{k < n: k\in A \} |}{n}  = 1.$$
It is an elementary matter to check that $\Fil_d$ is a \zwakkePfilter, but not a\linebreak \sterkePfilter, nor a rapid$^+$ filter. The notion of convergence corresponding to this filter is also known as statistical convergence.
\end{example}
    
\begin{example}$\;$\\
    To every $f: \Nat \to \mathbb{R}_{>0}$ with $\lim_{n \to \infty} f(n) = 0$ and $\sum_{n\in \Nat} f(n) = \infty$ corresponds a filter~$\Fil_{\text{sum},f}$ which is dual to the summable ideal
    $$\mathcal{I}_{\text{sum},f} = \{ A\subseteq \Nat : \sum_{n\in A} f(n) <\infty \}.$$
    One checks easily that $\Fil_{\text{sum},f}$ is a \sterkePfilter\ which is not rapid$^+.$
\end{example}

\begin{example}$\;$\\
    Let $b_1,b_2:\Nat\to\Nat$ such that the map $n\mapsto (b_1(n),b_2(n))$ is a bijection from $\Nat$ onto $\Nat\times \Nat$. The filter $\Fil_p\cong \{\Nat \}\times\mathcal{C}$ consists of those sets $A\subseteq\Nat$ that satisfy $$(\forall k\in\Nat)\quad \{b_2(n): n \in A\cap b_1^{-1}\left[\{ k \}\right] \} \in \mathcal{C}.$$
    By noting that a set $A\subseteq\Nat$ is $\Fil_p$-stationary if and only if there exists $k\in\Nat$ such that $ A\cap b_1^{-1}\left[\{ k \}\right]$ is infinite, it is easily seen that $\Fil_p$ is both a \zwakkePfilter\ and a rapid$^+$ filter. However, the sequence $(A_k)_k$ with $$A_k=\{ n\in\Nat: b_1(n)\geq k  \} , $$
    forms a decreasing sequence of $\Fil_p$-stationary sets and witnesses that $\Fil_p$ is not a \sterkePfilter.
\end{example}    

\noindent
All three filters $\Fil_d$, $\Fil_{\text{sum},f}$ and $\Fil_p$ belong to the well-studied class of analytic P-filters, see e.g.~\cite{farah}.

\section{Sufficient conditions for Fr\'echet-UBP-filters}

A natural weaker variation of the uniform \mbox{$\Fil$-bound}\-ed\-ness principle is given by the following stationary uniform \mbox{$\Fil$-bound}\-ed\-ness principle.
\begin{definition}$\;$\\
Let $X,Y$ be topological vector spaces.\\ A sequence $(T_i)_i$ in $\mathcal{L}(X,Y)$ is {\bf{\mbox{$\Fil$-sta}\-tion\-ary-equicontinuous}} if for  every \mbox{$\Fil$-sta}\-tion\-ary set $I \subseteq \mathbb{N}$ and every \mbox{$0$-neigh}\-bour\-hood $V$ in $Y$, there exists a \mbox{$0$-neigh}\-bour\-hood $U$ in $X$ such that:
$$ (\exists J \in \Fil^+, J \subseteq I)(\forall i \in J) \quad T_i[U] \subseteq V.$$
We say that the {\bf stationary uniform \mbox{$\Fil$-bound}\-ed\-ness principle} holds for continuous linear maps from $X$ to $Y$, and denote this by {$UBP^{\text stat}_{\Fil}(X,Y)$}, whenever every pointwise \mbox{$\Fil$-bounded} sequence $(T_i)_i$ in $\mathcal{L}(X,Y)$ is \mbox{$\Fil$-sta}\-tion\-ary-equi\-con\-ti\-nuous.
We say that the stationary uniform \mbox{$\Fil$-bound}\-ed\-ness principle holds for $X$ if {$UBP^{\text stat}_{\Fil}(X,Y)$} holds for every locally convex space~$Y$.
We define $\mathcal{F}$ to be a {\bf {stationary} Fr\'echet-UBP-filter} if the stationary uniform \mbox{$\Fil$-bound}\-ed\-ness principle holds for every Fr\'echet space~$X.$ {\bf {Stationary} Banach-UBP-filters} are defined in the analogous way.
\end{definition}

\noindent
We will find that, in several ways, the stationary uniform \mbox{$\Fil$-bound}\-ed\-ness principle relates to the uniform \mbox{$\Fil$-bound}\-ed\-ness principle as {\zwakkePfilter s} relate to {\sterkePfilter s}.
\bigskip

\noindent
In this section, it will be proved that the previously introduced combinatorial properties for filters can be used to express sufficient conditions for the uniform \mbox{$\Fil$-bound}\-ed\-ness principles to hold for Fr\'echet spaces.
In particular, we will prove the following theorem.

\begin{theorem}\label{Frechetsuf}$\,$
    \vspace{-2pt}
    \begin{enumerate}
        \item[(a)] Every rapid$^+$ {\zwakkePfilter} is a {stationary} Fr\'echet-UBP-filter.
        \item[(b)] Every rapid$^+$ {\sterkePfilter} is a Fr\'echet-UBP-filter.
    \end{enumerate}
\end{theorem}

\noindent
The theory of uniform boundedness principles for (locally convex) topological vector spaces is traditionally phrased in terms of barrels and barrelled spaces. To stay in line with this case where $\Fil$ is equal to the Fr\'echet filter~$\mathcal{C}$, we first introduce an \mbox{$\Fil$-analogue} for the concept of barrelled space.
Recall that a subset $B$ of a topological vector space~$X$ is a barrel when it is absorbing, balanced, closed and convex (our terminology here is in line with \cite{treves}).

\begin{definition}$\;$\\
An \mbox{$\Fil$-barrel}-system in $X$ is a sequence $(B_i)_i$ consisting of barrels $B_i$ in $X$ with the property that for every $x \in X,$ there is $t>0$ such that
$$tx \in B_i \quad (\forall_{\!\mathcal{F}}\, i \in \Nat).$$
\end{definition}
\smallskip

\noindent
In the following definition, we give the obvious meanings to pointwise \mbox{$\Fil$-bound}\-ed\-ness and \mbox{$\Fil$-equi}\-con\-tin\-uity of a sequence of continuous seminorms.

\begin{definition}$\;$\\
    Let $X$ be a topological vector space and $(p_i)_i$ a sequence of continuous seminorms on $X$, let $(B_i)_i$ be the corresponding sequence of barrels with
    $B_i = \{ x \in X : p_i(x) \leq 1\}.$\\ The sequence $(p_i)_i$ is called pointwise \mbox{$\mathcal{F}$-bounded} if $(B_{i})_i$ is an $\Fil$-barrel-system.\\
    The sequence $(p_i)_i$ is called \mbox{$\Fil$-equi}\-con\-tin\-uous if there exists a \mbox{$0$-neigh}\-bour\-hood $U$ in $X$ such that
    $$U \subseteq B_{i} \quad (\forall_{\!\mathcal{F}}\, i \in \Nat).$$
    The sequence $(p_i)_i$ is called $\Fil$-stationary-equicontinuous if for every \mbox{$\Fil$-sta}\-tion\-ary set $I \subseteq \mathbb{N}$ there exists a \mbox{$0$-neigh}\-bour\-hood $U$ in $X$ and an \mbox{$\Fil$-sta}\-tion\-ary $J \subseteq I$ such that
    $$U \subseteq B_i \quad (\forall i \in J).$$
\end{definition}

\begin{definition}$\;$\\
A topological vector space~$X$ is \mbox{$\Fil$-bar}\-relled if for every \mbox{$\Fil$-barrel}-system $(B_i)_i$ in $X$ there exists a \mbox{$0$-neigh}\-bour\-hood $U$ in $X$ such that
$$U \subseteq B_i \quad (\forall_{\!\mathcal{F}}\, i \in \Nat).$$
\end{definition}

\begin{definition}$\;$\\
A topological vector space~$X$ is {stationarily} \mbox{$\Fil$-bar}\-relled if for every \mbox{$\Fil$-barrel}-system $(B_i)_i$ in $X$ and for every \mbox{$\Fil$-sta}\-tion\-ary set $I,$ there exists a \mbox{$0$-neigh}\-bour\-hood $U$ in $X$ and an \mbox{$\Fil$-sta}\-tion\-ary $J \subseteq I$ such that
$$U \subseteq B_i \quad (\forall i \in J).$$
\end{definition}

\noindent
Note that if $\mathcal{F}$ coincides with the Fr\'echet filter~$\mathcal{C}$ of cofinite sets, the concepts of \mbox{$\Fil$-bar}\-relled space and {stationarily} \mbox{$\Fil$-bar}\-relled space do indeed coincide with the classical notion of barrelled space.\smallskip

\noindent
Just as in the case where $\Fil$ equals the Fr\'echet filter~$\mathcal{C}$, we find that continuous linear maps from \mbox{$\Fil$-bar}\-relled spaces to locally convex spaces satisfy uniform \mbox{$\Fil$-bound}\-ed\-ness principles.

\begin{lemma}$\;$ \label{BarrelleddanUBP}\\
For every \mbox{$\Fil$-bar}\-relled space~$X$ and every locally convex space~$Y$, $UBP_{\Fil}(X,Y)$ holds.\\
For every {stationarily} \mbox{$\Fil$-bar}\-relled space~$X$ and every locally convex space~$Y$,\linebreak \mbox{$UBP^{\text stat}_{\Fil}(X,Y)$} holds.
\end{lemma}

\begin{proof}
We prove the first statement, the proof of the second statement is analogous, and both proofs are entirely analogous to the proof of the uniform boundedness principle for barrelled spaces.\\
Let $(T_i)_i$ be a pointwise \mbox{$\mathcal{F}$-bounded} sequence in $\mathcal{L}(X,Y)$. Let $U$ be an arbitrary \mbox{$0$-neigh}\-bour\-hood in $Y$, we may assume that $U$ is a barrel. Because every $T_i$ is linear and continuous, $T_i^{-1}[U]$ is a barrel for every $i$. By pointwise \mbox{$\Fil$-bound}\-ed\-ness, $(T_i^{-1}[U])_i$ is an \mbox{$\Fil$-barrel}-system. Using \mbox{$\Fil$-bar}\-relled\-ness, we find some \mbox{$0$-neigh}\-bour\-hood $V$ in $X$ such that $(\forall_{\!\mathcal{F}}\,  i\in \mathbb{N})$ $V\subseteq T_i^{-1}[U]$, so we have proved \mbox{$\Fil$-equi}\-con\-tin\-uity of $(T_i)_i$.\EndProof
\end{proof}

\begin{remark}
It follows from the given definitions that if $X$ is a barrelled space, then $X$ is \mbox{$\Fil$-bar}\-relled if and only if the uniform \mbox{$\Fil$-bound}\-ed\-ness principle for seminorms holds for $X$, i.e.\ every pointwise \mbox{$\Fil$-bounded} sequence of continuous seminorms on $X$ is \mbox{$\Fil$-equi}\-con\-tin\-uous.
Analogously, if $X$ is barrelled, then $X$ is {stationarily} \mbox{$\Fil$-bar}\-relled if and only if the stationary uniform \mbox{$\Fil$-bound}\-ed\-ness principle for seminorms holds for $X$, i.e.\ every pointwise \mbox{$\Fil$-bounded} sequence of continuous seminorms on $X$ is \mbox{$\Fil$-sta}\-tion\-ary-equicontinuous.
\end{remark}

\noindent
In Theorem~\ref{Frechetbarrelled} below, we will prove that if $\Fil$ is a rapid$^+$ (weak) {\sterkePfilter}, then every Fr\'echet space is ({stationarily}) \mbox{$\Fil$-bar}\-relled. Our proof of Theorem~\ref{Frechetbarrelled} uses and extends several ideas found in the elementary proof for the uniform boundedness principle given by Sokal in \cite{sokal}. In this proof, we will repeatedly make use of a general non-empty intersection principle for completely metrizable topological vector spaces which we now describe first.\\
Our situation of interest will be the following. Suppose given a completely metrizable topological vector space~$X$ together with two countable bases $(V_n)_n$ and $(A_n)_n$ for the filter of \mbox{$0$-neigh}\-bour\-hoods of $X.$ Suppose that for all integers $l\geq 0$, $n \geq {1}:$
\begin{equation}
A_n + A_{n+1} + \ldots + A_{n+l} \subseteq V_n.
\tag{L1}
\end{equation}

\noindent
We will additionally assume that all neighbourhoods $V_n$ are closed.\hfill (L2)\label{eqn:L2}
\\Now, let $(C_n)_n$ be a sequence of subsets of $X$ and suppose that {we} wish to show that the intersection $\bigcap_{n\geq {1}} (C_n + V_{n+1})$ is non-empty. The following lemma provides a means of deriving this conclusion, at least when the sequences $(C_n)_n$ and $(A_n)_n$ have the following property:
\begin{align}
\begin{split}
&C_{1} \neq \emptyset.\\
&(\forall n\geq {2})(\forall x \in C_{n-1})((x+A_n) \cap C_{n} \not = \emptyset).
\end{split}
\tag{L3}
\end{align}
We call a sequence $(C_n)_n$ with this property $(A_n)_n$-connected.

\begin{lemma}$\;$ \label{nietlegedoorsnedepools}\\
If $X$ is a completely metrizable topological vector space with $(V_n)_n,$ $(A_n)_n$ two bases for the filter of \mbox{$0$-neigh}\-bour\-hoods of $X$ that satisfy (L1) and (L2), {then} the intersection $\bigcap_{n\geq {1}} (C_n + V_{n+1})$ is non-empty for every  sequence $(C_n)_n$ of subsets of $X$ that is $(A_n)_n$-connected (i.e.\ that satisfies (L3)).
\end{lemma}
\begin{proof}
We recursively construct a sequence $(x_k)_k$ in $X$.\\ Choose $x_{1}\in C_{1}$ arbitrary. For $k\geq {2}$, choose $x_k\in (x_{k-1}+A_k)\cap C_k$, it follows from (L3) that this is always  possible.
From this construction, it follows that
$$x_{k+l} - x_k \in A_{k+1} + \ldots + A_{k+l} \subseteq V_{k+1}$$
for every $l \in \Nat$ and every $k \geq {1}$.\\
Because $(V_k)_k$ is a \mbox{$0$-neigh}\-bour\-hood base for the complete space~$X$, we find that $x_k \to x$ for some $x\in X$.
Next, $(V_k)_k$ consisting of closed sets, we can take limits in $ x_{k+l} - x_k \in  V_{k+1}$ to find that $x\in x_k+V_{k+1}$ for every $k \geq {1}$.\\ Hence, $x\in\displaystyle \bigcap_{k\geq {1}} (x_k+V_{k+1})\subseteq \bigcap_{k\geq {1}} (C_k + V_{k+1}).$
\EndProof
\end{proof}

\noindent
We will also need the following two lemmas that both concern the local behaviour of seminorms.

\begin{lemma}[\cite{sokal}]$\;$ \label{lemSokal} \\
Let $X$ be a vector space and $U$ a symmetric subset (i.e., $U=-U$) of $X.$\\
If $p$ is a seminorm on $X$, then
$$\sup_{y \in x+ rU} p(y) \geq r \sup_{y \in U} p(y)\qquad \forall x\in X,\forall r\in \mathbb{R}_{>0} .$$
\end{lemma}
\begin{proof}
We can argue as follows:
\begin{align*}
\sup_{y \in x+ rU} p(y) &\geq  \frac{1}{2}\sup_{y \in U}( p(x+ry) + p(x-ry))\\
&\geq \frac{1}{2} \sup_{y \in U} p((x+ry) - (x-ry)) = r \sup_{y \in U} p(y).
\end{align*}
\EndProof
\end{proof}

\noindent
If $p$ is a seminorm on a vector space~$X,$ the following notation will be employed to denote the closed unit seminorm-ball defined by $p:$
$$B_p := \{ x \in X : p(x) \leq 1\} .$$

\begin{lemma}$\;$\label{normengeljkgroot}\\
Let $X$ be a vector space, and $p, q_1,\ldots,q_k$ seminorms on $X$.\\
If $q_1,\ldots,q_k$ are all unbounded on $B_p$, then $\min(q_1,\ldots,q_k)$ is unbounded on $B_p$.
\end{lemma}
\begin{proof}
Because each of the seminorms $q_i$ is unbounded on $B_p,$ there exists for every $i \in \{1, \ldots, k\}$ a sequence $(x_n^i)_n$ with $q_i(x_n^i)\geq 1$ and  $x_n^i\stackrel{p}{\to} 0$ as $n \to \infty$.\\
Now, as an auxiliary means for the rest of the proof, fix an arbitrary free ultrafilter $\Ul \subseteq P(\Nat).$
Consider for every $i$ the finite-dimensional vector space $V_i\leq \mathbb{R}^k$ defined by
$$V_i:=\{  (r_1,\ldots,r_k) \in \mathbb{R}^k : \displaystyle \lim_{\Ul,n} q_i(r_1x_n^1+\ldots+r_kx_n^k)=0 \}.$$ Since the vector $(0,\ldots,0,1,0,\ldots,0)$ (with $1$ on place $i$) is not contained in $V_i$, every $V_i$ is a proper subspace of $\mathbb{R}^k$. It follows that $\displaystyle \mathbb{R}^k\neq \bigcup_{i=1}^k V_i$, so we can select $(r_1,\ldots,r_k)\in \displaystyle\bigcap_{i=1}^k V_i^c$. \\
Set $z_n:=r_1x_n^1+\ldots+r_kx_n^k$. From the choice of the scalars $(r_1,\ldots,r_k)$ and the definition of an ultrafilter limit, it follows that there exists $\varepsilon >0$ and $U\in \Ul$ such that for every $i \in \{1, \ldots, k\}$ and every $n \in U$,  $q_i(z_n)>\varepsilon$.\\
Thus, $\min(q_1,\ldots,q_k)(z_n)>\varepsilon$ for every $n \in U$, while $z_n\stackrel{p}{\to} 0$ as $n \to \infty.$ Hence, $\min(q_1,\ldots,q_k)$ is unbounded on $B_p$.
\EndProof
\end{proof}

\begin{lemma}$\;$\label{puntsgewijsbegrensddanookopomgeving}\\
If $X$ is a Fr\'echet space, $\Fil$ a {\sterkePfilter} and $(q_i)_i$ a pointwise \mbox{$\Fil$-bounded} sequence of continuous seminorms on $X,$ then there exists a \mbox{$0$-neigh}\-bour\-hood $U$ in $X$ and $F \in \Fil$ such that $q_i$ is bounded on $U$ whenever $i \in F$ (with bounds possibly depending on $i$).
\end{lemma}

\begin{proof}
Fix a base $(p_m)_m$ of continuous seminorms for $X$ with $p_m \leq p_{m+1}$ for every $m.$
Because every $q_i$ is a continuous seminorm on $X,$ there exists for every~$i$ a corresponding integer $m$ such that $q_i$ is bounded on the seminorm-ball $B_{p_m}.$ As a consequence, the following function is well defined:
$$l: \Nat \to \Nat: l(i) = \min\{m\in \Nat: q_i \text{ is bounded on } B_{p_{m+1}} \}.$$
The assertion of the lemma is established if we can show that $l$ is bounded on an element of $\Fil.$ Because $\Fil$ is a {\sterkePfilter}, we can apply Lemma~\ref{alternatievekaraktrerisatiePx} to find that it suffices to prove that the function $l$ is not finite-to-one on any \mbox{$\Fil$-sta}\-tion\-ary set.\\
Striving for contradiction, let's suppose that $l$ is in fact finite-to-one on the \mbox{$\Fil$-sta}\-tion\-ary set $I.$
For every integer $k\geq 0$, set
$$ \sigma_k := \sup \{ q_i(x) : i \in I, l(i) = k \text{ and } x\in B_{p_{k+1}}   \}.$$
It follows from the definition of $l(i)$ together with $l$ being finite-to-one on $I,$ that $\sigma_k < \infty$.\\
Next, we wish to apply Lemma~\ref{nietlegedoorsnedepools} to the system $(A_k)_k$, $(V_k)_k$, $(C_k)_k$ defined by:
\begin{align*}
A_k&=\{ x\in X : p_k(x) \leq 3^{-k}  \}\\
V_k&= \{ x\in X : p_k(x) \leq \frac{3}{2} 3^{-k} \}\\
C_k&= \{x \in X: (\forall i \in I)(l(i)=k \Rightarrow q_i(x) \geq k + \frac{3\sigma_k}{2}) \}.
\end{align*}
It is easily verified that the filter bases $(A_k)_k$ and $(V_k)_k$ indeed satisfy the conditions (L1) and (L2). Before checking (L3), note that for every $k \geq 1,$ the set $\{q_i : l(i)=k, i \in I \}$ contains finitely many seminorms which are all unbounded on $B_{p_{k}}$.  Hence, it follows directly from Lemma~\ref{normengeljkgroot} that $C_k$ is non-empty. We prove that in addition $(y+A_k)\cap C_k$ is non-empty for every $k \geq 2$ and every $y\in X$, so that $(C_k)_k$ is indeed $(A_k)_k$-connected. Given $y \in X$ and $k\geq 2$ it follows from a second application of Lemma~\ref{normengeljkgroot} that there is $x \in \frac{1}{3^k}B_{p_{k}}$ such that for every $i \in I$ with $l(i)=k$ the following is true
$$q_i(x) \geq k + \frac{3}{2}\sigma_k + \max\{q_j(y) : j \in I, l(j) =k\}.$$ Then $x + y$ belongs to $(y+A_k)\cap C_k.$
It follows that we can apply Lemma~\ref{nietlegedoorsnedepools} and find that there exists $x \in \bigcap_{k \geq 1} (C_k + V_{k+1}).$
Then for every $i\in I$ and $k \geq 1$ with $l(i) = k$, we have:
\begin{align*}
q_i(x) &\geq \inf\{q_i(y) : y \in C_k \} - \sup\{q_i(y) : y \in V_{k+1} \}\\ &\geq k + \frac{3\sigma_k}{2} - \sup\{q_i(y) : y \in V_{k+1} \}\\ &\geq k + \frac{3\sigma_k }{2}(1-3^{-k-1})\to \infty,
\end{align*}
as $k \to \infty.$  Because $l$ is also finite-to-one on $I$, it follows that $(q_i(x))_i$ tends to~$+\infty$ on the \mbox{$\Fil$-sta}\-tion\-ary set $I$  and consequently that $(q_i)_i$ can not be pointwise \mbox{$\Fil$-bounded}.
\EndProof
\end{proof}

\begin{lemma}$\;$\label{puntsgewijsbegrensddanookopomgevingstat}\\
If $X$ is a Fr\'echet space, $\Fil$ a {\zwakkePfilter} and $(q_i)_i$ a pointwise \mbox{$\Fil$-bounded} sequence of continuous seminorms on $X,$ then for every \mbox{$\Fil$-sta}\-tion\-ary set $I$, there exists a \mbox{$0$-neigh}\-bour\-hood $U$ in $X$ and an \mbox{$\Fil$-sta}\-tion\-ary subset $J$ of $I$ such that $q_i$ is bounded on $U$ whenever $i \in J$ (with bounds possibly depending on $i$).
\end{lemma}
\begin{proof}
As in Lemma~\ref{puntsgewijsbegrensddanookopomgeving},
$$l: \Nat \to \Nat: l(i) = \min\{m \in \Nat: q_i \text{ is bounded on } B_{p_{m+1}} \},$$
is well defined.
Let $I$ be an arbitrary \mbox{$\Fil$-sta}\-tion\-ary set. We will apply Lemma~\ref{alternatievekaraktrerisatieP+} to the mapping $l_I: \Nat \to \Nat,$ defined by $l_I(i) = l(i)$ if $i\in I$ and $l_I(i) = i$ for every $i\in I^c.$\\
If $l_I$ is bounded on an \mbox{$\Fil$-sta}\-tion\-ary set, then it is also bounded on an \mbox{$\Fil$-sta}\-tion\-ary subset of $I$ and this would lead directly to the conclusion in the lemma.
Because $\Fil$ is a {\zwakkePfilter} it suffices to prove that the function $l_I$ is not finite-to-one on any \mbox{$\Fil$-sta}\-tion\-ary subset of $I.$
Suppose in desire of contradiction that $l$ is finite-to-one on the \mbox{$\Fil$-sta}\-tion\-ary subset $J$ of~$I.$
Proceeding exactly as in Lemma~\ref{puntsgewijsbegrensddanookopomgeving}, but with $J$ in the role of $I$, one uses Lemma~\ref{nietlegedoorsnedepools} to produce under this assumption $x\in X$ such that $(q_i(x))_i$ tends to $+\infty$ on the \mbox{$\Fil$-sta}\-tion\-ary set~$J$. This leads to contradiction with the pointwise \mbox{$\Fil$-bound}\-ed\-ness of~$(q_i)_i$.
\EndProof
\end{proof}

\noindent
We are now ready to give the proof of Theorem~\ref{Frechetbarrelled}. Note that Theorem~\ref{Frechetsuf} then follows by combining Theorem~\ref{Frechetbarrelled} and Lemma~\ref{BarrelleddanUBP}.

\begin{theorem}$\;$\label{Frechetbarrelled}
    \begin{enumerate}
        \item[(a)] If $X$ is a Fr\'echet space and $\Fil$ is a rapid$^+$ {\sterkePfilter} then $X$ is \mbox{$\Fil$-bar}\-relled.
        \item[(b)] If $X$ is a Fr\'echet space and $\Fil$ is a rapid$^+$ {\zwakkePfilter} then $X$ is\linebreak {stationarily} \mbox{$\Fil$-bar}\-relled.
    \end{enumerate}
\end{theorem}

\begin{proof}
Let $X$ be a Fr\'echet space and let $\Fil$ be a filter.\\
Fix a base $(p_m)_m$ of continuous seminorms for $X$ with $p_m \leq p_{m+1}$ for every~$m.$
Let an arbitrary \mbox{$\Fil$-barrel}-system $(B_i)_i$ in $X$ be given.
Denote by $q_i$ the Minkowski-functional corresponding to the barrel $B_i$. Because $X$ is barrelled, every $q_i$ is a continuous seminorm on $X$ and because $(B_i)_i$ forms an \mbox{$\Fil$-barrel}-system, the sequence $(q_i)_i$ is pointwise \mbox{$\Fil$-bounded}.\\
We first prove the following claim:
\begin{claim}{If $(q_i)_i$ is a pointwise \mbox{$\Fil$-bounded} sequence of seminorms on $X$ and $(a_k)_k$ is a sequence of non-negative integers satisfying
  $$ 4 ^k \leq  \sup\{ q_{a_k}(x) : p_k(x) \leq 1 \} < \infty $$ for every $k$,
   then  $\{a_k : k \in \Nat\}$ can not be \mbox{$\Fil$-sta}\-tion\-ary.}
\end{claim}

\begin{proof}[Proof of Claim]\renewcommand{\qedsymbol}{}
{
Suppose that we could find such a sequence $(a_k)_k$ of non-negative integers with $\{a_k : k \in \Nat\}$ \mbox{$\Fil$-sta}\-tion\-ary.\\
Then, we could apply Lemma~\ref{nietlegedoorsnedepools} to the system $(A_k)_k$, $(V_k)_k$, $(C_k)_k$ defined by:
\begin{align*}
A_k &=\{ x\in X : p_k(x) \leq 3^{-k}  \}\\
V_k &= \{ x\in X : p_k(x) \leq \frac{3}{2} 3^{-k} \}\\
C_k &= \{ x\in X : q_{a_k}(x) \geq \frac{2}{3^{k+1}} \sup\{ q_{a_k}(x) : p_k(x) \leq 1  \}\}.
\end{align*}
It is again easily verified that the filter bases $(A_k)_k$ and $(V_k)_k$ satisfy the conditions (L1) and (L2) from Lemma~\ref{nietlegedoorsnedepools} and it follows from a direct application of Lemma~\ref{lemSokal} that $(C_k)_k$ is $(A_k)_k$-connected. Therefore, we can apply Lemma~\ref{nietlegedoorsnedepools} to find some $x\in \bigcap_{k\geq 1} (C_k + V_{k+1}).$
Then:
\begin{align*}
q_{a_k}(x) &\geq \inf\{q_{a_k}(y) : y \in C_k \} - \sup\{q_{a_k}(y) : y \in V_{k+1} \}\\ &\geq  \frac{2}{3^{k+1}} \sup\{ q_{a_k}(x) : p_k(x) \leq 1  \} - \frac{3}{2}\ 3^{-(k+1)} \sup\{ q_{a_k}(x) : p_{k+1}(x) \leq 1  \} \\ &\geq \frac{1}{6}\, 3^{-k} \sup\{ q_{a_k}(x) : p_k(x) \leq 1  \} \geq  \frac{1}{6} \left(\frac{4}{3}\right)^k \to \infty,
\end{align*}
as $k \to \infty.$  Hence, $(q_i(x))_i$ tends to $+\infty$ on the \mbox{$\Fil$-sta}\-tion\-ary set $\{a_k : k \in \Nat\},$ so that $(q_i)_i$ is not pointwise \mbox{$\Fil$-bounded}.
This concludes the proof of the claim.}
\end{proof}

\noindent
We proceed to prove statements (a) and (b) separately.
\smallskip

\noindent
(a) Let the filter $\Fil$ now be a rapid$^+$ {\sterkePfilter}. To prove part (a) of the theorem, it suffices to prove that the pointwise \mbox{$\Fil$-bounded} sequence $(q_i)_i$ of seminorms is \mbox{$\Fil$-equi}\-con\-tin\-uous. If follows from Lemma~\ref{puntsgewijsbegrensddanookopomgeving} that there is a \mbox{$0$-neigh}\-bour\-hood $U$ in $X$ and $F \in \Fil$ such that $q_i$ is bounded on $U$ whenever $i \in F$. Hence, by altering the sequence $(q_i)_i$ on a non-\mbox{$\Fil$-sta}\-tion\-ary set and renumbering the seminorms $(p_m)_m$ if necessary, it can be assumed that $\sup\{ q_{i}(x) : p_m(x) \leq 1  \}$ is finite for every $i,m \in \Nat.$  Suppose now that the sequence $(q_i)_i$ is not \mbox{$\Fil$-equi}\-con\-tin\-uous.\\
Set $X_k:=\{  i \in \Nat : \sup\{ q_{i}(x) : p_k(x) \leq 1  \} > 4^k  \}$. Then $(X_k)_k$ is a decreasing sequence of \mbox{$\Fil$-sta}\-tion\-ary sets and because $\Fil$ is a {\sterkePfilter}, there is an \mbox{$\Fil$-sta}\-tion\-ary $B$ such that  $B \subseteq^\ast X_k$ for every $k.$  Because $\Fil$ is rapid$^+,$ there exist $a_k\in X_k$ such that $\{ a_k : k\in \Nat   \}$ is \mbox{$\Fil$-sta}\-tion\-ary. But then the above claim can be used to reach a contradiction with the pointwise \mbox{$\Fil$-bound}\-ed\-ness of the sequence  $(q_i)_i$.
\smallskip

\noindent
(b) Let the filter $\Fil$ now be a rapid$^+$ {\zwakkePfilter}. Let $I$ be an arbitrary \mbox{$\Fil$-sta}\-tion\-ary set. To prove part (b) of the theorem, it suffices to prove that there exists an \mbox{$\Fil$-sta}\-tion\-ary set $J \subseteq I$ and a \mbox{$0$-neigh}\-bour\-hood $U$ in $X$ such that
$$\sup\{ q_j(x) : j \in J, x \in U \} $$ is finite. Suppose ($\ast$) that such a $J$ and $U$ do not exist. It follows from Lemma~\ref{puntsgewijsbegrensddanookopomgevingstat} that we can assume that $\sup\{ q_{i}(x) : p_m(x) \leq 1 \}$ is finite for every $i \in I$, $m \in \Nat.$ Indeed, since $(q_i)_i$ is pointwise \mbox{$\Fil$-bounded}, this can be accomplished (without affecting generality of the proof) by replacing $I$ by some \mbox{$\Fil$-sta}\-tion\-ary subset of $I$ and by renumbering the seminorms $(p_m)_m$, if necessary.\\
Set $X_k:=\{  i \in I : \sup\{ q_{i}(x) : p_k(x) \leq 1  \} > 4^k  \}$. By the assumption ($\ast$), none of the sets $X_k^c \cap I$ can be \mbox{$\Fil$-sta}\-tion\-ary. Then $(X_k \cup I^c)_k$ is a decreasing sequence of elements of $\Fil$ and because $\Fil$ is a {\zwakkePfilter}, there is an \mbox{$\Fil$-sta}\-tion\-ary $B \subseteq I$ such that  $B \subseteq^\ast X_k \cup I^c$ for every $k.$ It follows that also $B \subseteq^\ast X_k$ for every $k.$  Because $\Fil$ is rapid$^+,$ there exist $a_k\in X_k$ such that $\{ a_k : k\in \Nat   \}$ is \mbox{$\Fil$-sta}\-tion\-ary. But then the above claim can again be used to reach a contradiction with the pointwise \mbox{$\Fil$-bound}\-ed\-ness of the sequence  $(q_i)_i$.
\EndProof
\end{proof}

\begin{remark}
    For the reader who is solely interested in Banach spaces instead of the more general Fr\'echet spaces, the previous proof can be considerably simplified. Indeed, one finds that whenever $X$ is Banach, both Lemma~\ref{puntsgewijsbegrensddanookopomgeving} and Lemma~\ref{puntsgewijsbegrensddanookopomgevingstat} reduce to trivialities. Consequently, to prove that every Banach space is ({stationarily}) \mbox{$\Fil$-bar}\-relled whenever $\Fil$ is a rapid$^+$ (weak) {\sterkePfilter}, one can directly repeat the proof of Theorem~\ref{Frechetbarrelled}, needing only Lemma~\ref{nietlegedoorsnedepools} and Lemma~\ref{lemSokal} at hand.
\end{remark}

\begin{remark}\label{nieuwe opm}
Theorem~\ref{Frechetsuf} can be applied in similar fashion as the uniform boundedness principle to derive \mbox{$\Fil$-dependent} results in (functional) analysis. We illustrate this with two examples.\smallskip
\end{remark}

\noindent
{\it Weakly \mbox{$\Fil$-con}\-ver\-gent sequences}\\
Recall the following consequence of the uniform boundedness principle which is of special value in Banach space theory: \smallskip

\begin{minipage}[c]{0,95\textwidth}
if $X$ is Banach, then every $\sigma(X, X')$-bounded sequence $(x_n)_n$  in $X$ is norm-bounded. In particular every weakly convergent sequence $(x_n)_n$ in $X$ is norm-bounded.
\end{minipage}
\smallskip

\noindent
Note that since $X$ embeds linear-isometrically in its bidual, the above statement follows indeed by a direct application of the uniform boundedness principle.\\
Using the same reasoning one arrives at the following consequence of Theorem~\ref{Frechetsuf}.

 \begin{corollary}$\;$\label{remark1}\\
     Let $X$ be a Banach space and let $\Fil$ be a rapid$^+$ {\sterkePfilter}, then every sequence $(x_n)_n$ in~$X$ which is \mbox{$\Fil$-bounded} with respect to the $\sigma(X,X')$-topology is also \mbox{$\Fil$-bounded} in norm. In particular, if $(x_n)_n$ is weakly \mbox{$\Fil$-con}\-ver\-gent\footnote{$(x_n)_n$ is defined to be weakly $\Fil$-convergent to $x\in X$ whenever $(\varphi(x_n))_n$ is $\Fil$-convergent to $\varphi(x)$ for every $\varphi \in X'.$}, then $(x_n)_n$ is \mbox{$\Fil$-bounded} in norm.
 \end{corollary}

\noindent
The second conclusion of Corollary \ref{remark1} can be read in the following way: every weakly \mbox{$\Fil$-con}\-ver\-gent sequence $(x_n)_n$ in $X$ coincides with some bounded sequence $(y_n)_n$ on a set of indices which is contained in $\Fil.$\\
 It is instructive to note that there exist filters $\Fil$ for which this fails strongly in the following sense:\smallskip
 
 \begin{minipage}[c]{0,95\textwidth}
  every infinite-dimensional Banach space~$X$ contains some weakly \mbox{$\Fil$-con}\-ver\-gent sequence $(x_n)_n$ with the property that no (infinite) subsequence of $(x_n)_n$ is bounded.\end{minipage} \smallskip
 
  \noindent
  Indeed, an example of such a filter is the asymptotic density filter~$\Fil_d$.
 In fact, a characterisation of the (free) filters with the above property has been established in \cite[Theorem 3.2 and Remark 1]{ganichev}, where it is shown that the following three are equivalent
\begin{itemize}
\item[(i)] there exists an infinite-dimensional Banach space~$X$ which contains some weakly \mbox{$\Fil$-con}\-ver\-gent sequence $(x_n)_n$ with the property that no (infinite) subsequence of $(x_n)_n$ is bounded,\item[(ii)] every infinite-dimensional Banach space~$X$ contains some weakly \mbox{$\Fil$-con}\-ver\-gent sequence $(x_n)_n$ with the property that no (infinite) subsequence of $(x_n)_n$ is bounded,\item[(iii)] There exists a positive real sequence $(c_n)_n$ satisfying both $\displaystyle\lim_{n \to \infty}c_n= + \infty$ and $c_n = O(n)$, such that $\displaystyle\{ A\subseteq \Nat: \lim_{n \to \infty} \frac{|\{k<n : k \notin A \}|}{ c_n} = 0\}\subseteq \Fil.$
\end{itemize}
\bigskip

\noindent
{\it $\Fil$-convergence of Fourier-sums}\\
By applying Theorem~\ref{Frechetsuf} in the Banach space $C(\mathbb{T})$ of functions continuous on the circle, equipped with the supremum norm, one proves the following strengthening of a well-known result (see e.g.\ \cite[5.11]{rudin2})
for $\Fil=\mathcal C$.

\begin{lemma}$\;$\\
    For every rapid$^+$ {\zwakkePfilter} $\Fil,$ there exists $f\in C(\mathbb{T})$ such that the corresponding sequence 
    $$s_n(f) = \sum _{k=-n}^{n}{\hat {f}}(k)$$
    of partial Fourier-sums at $0$ is not \mbox{$\Fil$-bounded}. In particular, the series $\sum _{k \in \mathbb{Z}}{\hat {f}}(k)$ is not \mbox{$\Fil$-con}\-ver\-gent.
\end{lemma}
\begin{proof}
    Note that the operator norms of the functionals $\Lambda_n :C(\mathbb{T}) \to \mathbb{R}: f \mapsto \sum _{k=-n}^{n}{\hat {f}}(k)$ tend to infinity as $n \to \infty.$ By consequence, the sequence $(\Lambda_n)_n$ is not $\Fil$-stationary-equicontinuous,
    for any filter~$\Fil.$ The assertion follows directly by applying Theorem~\ref{Frechetsuf}.\EndProof
\end{proof}

\section{Necessary conditions for Fr\'echet-UBP-filters}
\subsection{Fr\'echet-UBP-filters are {\sterkePfilter s}}
We now pass to the study of those properties of $\Fil$ which are necessary for $\Fil$ to be a Fr\'echet-UBP-filter or a {stationary} Fr\'echet-UBP-filter. In this section we will show that every ({stationary}) Fr\'echet-UBP-filter is a (weak) {\sterkePfilter}.\\
In fact, something stronger is true: it is enough for the uniform \mbox{$\Fil$-bound}\-ed\-ness principle to hold for (at least) one arbitrary infinite-dimensional Fr\'echet space~$X$ to imply that $\Fil$ is a \sterkePfilter. The proof is somewhat surprising in the sense that it separates the case where $X$ is normed and the case where $X$ is non-normable. In the latter case, which is treated in the following lemma, the proof is based on directly exploiting the non-normability assumption on $X$.
 
\begin{lemma}$\;$ \label{P  voor niet normeerbaar}\\
Let $X$ be an infinite-dimensional metrizable locally convex topological vector space which is not normable.
\begin{enumerate}
    \item If the uniform \mbox{$\Fil$-bound}\-ed\-ness principle holds for $X,$ then $\Fil$ is a {\sterkePfilter}.
    \item If the stationary uniform \mbox{$\Fil$-bound}\-ed\-ness principle holds for $X,$ then $\Fil$ is a {\zwakkePfilter}.
\end{enumerate}
\end{lemma}
\begin{proof}
We give the proof of the first statement, it is easily adapted to prove the second statement.
Suppose that $\Fil$ is not a {\sterkePfilter}. Using Lemma~\ref{alternatievekaraktrerisatiePx}, we find $\Gamma: \mathbb{N} \to \mathbb{N}$ with:
\begin{itemize}
\item $\Gamma$ is unbounded on every element of $\Fil$,
\item $\Gamma$ is not finite-to-one on any \mbox{$\Fil$-sta}\-tion\-ary set.
\end{itemize}
Let $(V_k)_k$ be a decreasing \mbox{$0$-neigh}\-bour\-hood base in $X$. Since $X$ is not normable, none of the neighbourhoods $V_k$ is bounded. Since $X$ is locally convex, it follows from the Mackey theorem that the original topology on $X$ and the weak topology $\sigma(X,X')$ share the same bounded sets. Therefore, none of the neighbourhoods $V_k$ is $\sigma(X,X')$-bounded. It follows that we can select for each $n \in \Nat$ a functional $\varphi_n \in X'$ which is unbounded on $V_n.$ We claim that the sequence $(\psi_n)_n$ with $\psi_n(x) := \frac{1}{n} \varphi_{\Gamma(n)}$ contradicts the uniform \mbox{$\Fil$-bound}\-ed\-ness principle. %
Indeed, it is clear that $\Gamma$ is finite-to-one on any set of the form  $\{n \in \Nat: |\psi_n(x)|> C \}$, with $C \in \mathbb{R}_{>0}$ and $x \in X.$ Therefore, all of these sets have to be non-\mbox{$\Fil$-sta}\-tion\-ary, hence the sets $\{n \in \Nat: |\psi_n(x)| \leq C \}$ have to belong to $\Fil$ and $(\psi_n)_n$ is pointwise \mbox{$\Fil$-bounded}. From the uniform \mbox{$\Fil$-bound}\-ed\-ness principle it would follow that there exist $k \in \Nat$ and $A \in \Fil$ such that
$$(\forall x \in V_k)(\forall n \in A)\quad |\psi_n(x)| \leq 1 .$$
Since $\Gamma$ is unbounded on $A,$ we can choose $n \in A$ such that $\Gamma(n) > k.$ Then:
$$ (\forall x \in V_k)\quad |\varphi_{\Gamma(n)}(x)| \leq n.$$
This contradicts the fact that $\varphi_{\Gamma(n)}$ is unbounded on $V_{\Gamma(n)} \subseteq V_k.$
 \EndProof
 \end{proof}
 
 \noindent
 In the remainder of this section we concentrate on infinite-dimensional normed spaces $X$ and show that the (stationary) uniform \mbox{$\Fil$-bound}\-ed\-ness principle cannot hold for $X$ when $\Fil$ is not a (weak) {\sterkePfilter}. For this purpose, we make use of the Josefson-Nissenzweig theorem which assures that the dual space of every Banach space $X$ is rich enough for counterexamples to the uniform \mbox{$\Fil$-bound}\-ed\-ness principle to exist.
 
\begin{lemma}$\;$\label{Pointwise bounded P}\\
    Let $X$ be an infinite-dimensional normed space, $\mathcal{F}$ a filter on $\Nat$ and $I\subseteq \Nat$ \mbox{$\Fil$-sta}\-tion\-ary. If $\Gamma : \Nat \to \Nat$ is not finite-to-one on any \mbox{$\Fil$-sta}\-tion\-ary subset of $I,$
    then there exists a pointwise \mbox{$\Fil$-bounded} sequence $(\varphi_l)_l$ in the dual space~$X'$ such that $\| \varphi_l \| = \Gamma(l)$ for every $l\in I.$
\end{lemma}
\begin{proof}
    Note that by passing to the completion of $X$ if necessary, we may assume that $X$ is complete.
    By the Josefson-Nissenzweig theorem (see for example \cite{behrends} for a complete proof), we can choose a sequence $(\psi_n)_n$ in the unit sphere of $X'$ which is $\sigma(X',X)$-null (i.e.\ $\psi_n(x) \to 0$ for every $x\in X$).
    Consider the sequence of bounded linear functionals 
    given by $$\varphi_l=\begin{cases}\Gamma(l)\psi_{l} & \text{when } l \in I, \\   0 & \text{otherwise}.\end{cases}$$
    We claim that for every $x\in X$ and every $C \in \mathbb{R}_{>0},$ $$\{ l\in\Nat: | \varphi_l(x) | \leq C   \}\in\mathcal{F}.$$
    Suppose not, then $\{  l\in\Nat: |\varphi_l(x)| > C \}$ is an \mbox{$\Fil$-sta}\-tion\-ary subset of $I$.\\
    Because $\Gamma$ is not finite-to-one on any \mbox{$\Fil$-sta}\-tion\-ary subset of $I,$ there exists $m \in \Nat$ such that
    $A_m :=  \{  l\in\Nat: |\varphi_l(x)| > C\,\wedge \, \Gamma(l) = m \}$ is an infinite set.
    But then
    $$C \leq  \lim_{l \in A_m, l \to \infty} |\varphi_l(x)| =   \lim_{l \in A_m, l \to \infty}  m |\psi_{l}(x)|  = 0,$$
    a contradiction which justifies our claim.
\EndProof
\end{proof}

\begin{lemma}$\;$\label{Nec P} \\
    Let $X$ be an infinite-dimensional normed space and $\mathcal{F}$ a filter on $\Nat.$
    \begin{enumerate}
        \item[(a)] If the uniform \mbox{$\Fil$-bound}\-ed\-ness principle holds for $X$, then $\Fil$ is a {\sterkePfilter}.
        \item[(b)] If the stationary uniform \mbox{$\Fil$-bound}\-ed\-ness principle holds for $X$, then $\Fil$ is a {\zwakkePfilter}.
    \end{enumerate}
\end{lemma}
\begin{proof}
    (a) Suppose that $\Fil$ is not a {\sterkePfilter}.\\ Using Lemma~\ref{alternatievekaraktrerisatiePx}, choose $\Gamma: \Nat \to\Nat$ such that:
    \begin{itemize}
        \item $\Gamma$ is unbounded on every element of $\Fil,$
        \item $\Gamma$ is not finite-to-one on any \mbox{$\Fil$-sta}\-tion\-ary subset of $\Nat.$
    \end{itemize}
    Using Lemma~\ref{Pointwise bounded P}, we find a pointwise \mbox{$\Fil$-bounded} sequence $(\varphi_l)_l$ in $X'$ such that $\| \varphi_l \| = \Gamma(l)$ for every $l.$ To reach a contradiction with the uniform \mbox{$\Fil$-bound}\-ed\-ness principle, it suffices to note that $(\varphi_l)_l$ can not be \mbox{$\Fil$-bounded}, because $\Gamma$ is not bounded on any element of $\Fil.$\\
    
    (b) Suppose that $\Fil$ is not a {\zwakkePfilter}.\\ Using Lemma~\ref{alternatievekaraktrerisatieP+}, choose $\Gamma: \Nat \to\Nat$ and an \mbox{$\Fil$-sta}\-tion\-ary set $I$ such that:
    \begin{itemize}
        \item $\Gamma$ is unbounded on every \mbox{$\Fil$-sta}\-tion\-ary set,
        \item $\Gamma$ is not finite-to-one on any \mbox{$\Fil$-sta}\-tion\-ary subset of $I.$
    \end{itemize}
    Using Lemma~\ref{Pointwise bounded P}, we find a pointwise \mbox{$\Fil$-bounded} sequence $(\varphi_l)_l$ in $X'$ such that $\| \varphi_l \| = \Gamma(l)$ for every $l\in I.$  Because $\Gamma$ is unbounded on every \mbox{$\Fil$-sta}\-tion\-ary set and $\| \varphi_l \| = \Gamma(l)$ for every $l \in I,$ we find that $(\varphi_l)_l$ is not bounded on any \mbox{$\Fil$-sta}\-tion\-ary subset of $I$. Since  $(\varphi_l)_l$ is pointwise \mbox{$\Fil$-bounded}, this is in contradiction with the stationary uniform \mbox{$\Fil$-bound}\-ed\-ness principle.\EndProof
\end{proof}

\noindent
One can note that the sequences constructed in the proofs of both Lemma~\ref{P voor niet normeerbaar} and Lemma~\ref{Pointwise bounded P} are not only pointwise \mbox{$\Fil$-bounded} but also pointwise \mbox{$\Fil$-con}\-ver\-gent to $0$. By also noting that $\Fil_d$ is not a \sterkePfilter, we arrive at the following Corollary~\ref{revision}.
We remark that in the case that $X$ is a Banach space, Corollary~\ref{revision} is a direct consequence of Theorem~1 in \cite{connor}.

\begin{corollary}$\;$\label{revision}\\
    For every infinite-dimensional metrizable locally convex topological vector space there exists a sequence $(\varphi_l)_l \in X'$ that is not \mbox{$\Fil_d$-equi}\-con\-tin\-uous, yet pointwise converges statistically to zero.
\end{corollary}

\subsection{Stationary Fr\'echet-UBP-filters are Rapid$^+$-filters}
To conclude the characterisation of ({stationary}) Fr\'echet-UBP-filters, we prove that every {stationary} Banach-UBP-filter is rapid$^+$. The following more general statement holds.

\begin{lemma}$\;$\label{nec rapid}\\
    Let $X$ be an infinite-dimensional normed space and $\mathcal{F}$ a filter on $\Nat.$ If the stationary uniform \mbox{$\Fil$-bound}\-ed\-ness principle holds for $X$, then $\Fil$ is a rapid$^ +$ filter.
\end{lemma}

\noindent
For this purpose, we adapt an argument used in the proof of Theorem 2.6 in~\cite{kadets}. The main tool in this argument is Dvoretzky's theorem on the existence of almost-euclidean sequences in normed spaces.
\\
\newline
Let $(X, \|\ \|)$ now be a normed space, let $\varepsilon > 0$. \\
Denote by $|a|_2$ the euclidean norm on $a\in \mathbb{R} ^{n+1}, |(a_0, \ldots ,a_n)|_2 := \sqrt{a_0^2 + \ldots + a_n^2}$.\\
We will call a finite sequence $x_m, \ldots, x_{m+n}$ $\varepsilon$-almost-euclidean whenever
\begin{align*}
    |(a_m, \ldots, a_{m+n})|_2^2 \leq \|\sum_{i=0}^n a_{m+i} x_{m+i} \| ^2\leq (1+\varepsilon) |(a_{m}, \ldots, a_{m+n})|_2^2,\\
    \qquad (\forall (a_m, \ldots, a_{m+n}) \in \mathbb{R} ^{n+1}).
\end{align*}

\begin{definition}$\;$\\
    Let $(I_n)_n$ be an increasing interval-partition of $I$, let $\varepsilon > 0$.\\
    A sequence $(x_n)_{n \in I}$ in $X$ is $\varepsilon$-almost-euclidean on $(I_n)_n$ if the finite sequence $(x_{j} : j \in I_n)$ is $\varepsilon$-almost-euclidean for every $n \in \mathbb{N}.$
\end{definition}

\begin{lemma}$\;$\label{lemma}\\
    Let $X$ be a normed space, $\Fil$ a filter on $\mathbb{N}$, $I\subseteq \Nat$ \mbox{$\Fil$-sta}\-tion\-ary, $\varepsilon>0$ and $(I_n)_n$ an increasing interval-partition of $I$.
    If
    \begin{itemize}
        \item[(a)] $\Fil$ is slow with respect to $(I_n)_n,$
        \item[(b)] $X'$ contains a sequence $(\varphi_n)_{n\in I}$ that is $\varepsilon$-almost-euclidean on $(I_n)_n,$
    \end{itemize}
    then $X'$ contains a pointwise \mbox{$\Fil$-bounded} sequence $(\psi_n)_n$ which is not bounded on any infinite subset of $I.$
\end{lemma}

\begin{proof} (Following closely the proof of \cite[Theorem 2.6]{kadets}.)\\
    Choose $f: \Nat \to I$ strictly increasing with $(I_n)_n = (I^f_n)_n.$
    Because $f$ is strictly increasing, $f$ has a non-decreasing left-inverse $g:I \to\mathbb{N}$. Define $\psi_n = \sqrt{g(n)} \varphi_n$ for $n \in I$ and $\psi_n = 0$ otherwise. For every $n \in I,$ the vector $\varphi_n$ is contained in an $\varepsilon$-almost-euclidean finite sequence and hence $\| \varphi_n \| \geq 1.$ Since $g$ is unbounded and non-decreasing, it follows that the sequence $(\psi_n)_n$ is not bounded (in $X'$) on any infinite subset of $I$. Next, take $x \in X$ with $\|x\|\le 1$ arbitrarily. We will prove that $A_r = \{k \in\Nat: |\psi_k(x)| \leq r  \}$ belongs to $\Fil,$ for certain $r\in \mathbb{R}_{>0}.$ Let $n$ be arbitrary, put $B_r := \Nat \setminus A_r$ and put $b_n = | B_r \cap I_n |.$ %
    \begin{claim}{ For a suitable choice of $r$ we have that $b_n \leq n$ for each $n\in \Nat.$}\end{claim}

    \begin{proof}[Proof of Claim]\renewcommand{\qedsymbol}{}
        {
        Consider $\chi_n = \sum_{i \in  B_r \cap I_n} \operatorname {sgn} (\psi_i(x)) \varphi_i.$\\ First note that
        \begin{equation}
        \chi_n(x) \geq \frac{b_n r}{\sqrt{g(f(n))}} = \frac{b_n r}{\sqrt{n}}.
        \tag{IE1}
        \end{equation}
         By $(\varphi_i : i\in I_n)$ being $\varepsilon$-almost-euclidean in $X',$ we have: $$\| \chi_n \|^2 \leq  (1+\varepsilon) b_n,$$
        hence
        \begin{equation}
        \chi_n(x) \leq \sqrt{(1+\varepsilon) b_n}.
        \tag{IE2}
        \end{equation}
        By combining inequalities (IE1) and (IE2) and selecting $r = \sqrt{1+\varepsilon},$ we find at once,
        $$ \sqrt{b_n} \leq \sqrt{n}.$$}
    \end{proof}
    
    \noindent
    Since $B_r \subseteq I$ and $\Fil$ is slow with respect to $(I_n)_n,$ we have found $r$ such that $B_r$ is non-\mbox{$\Fil$-sta}\-tion\-ary and hence $A_r \in \Fil.$
\EndProof
\end{proof}

\noindent
The existence of $\varepsilon$-almost-euclidean sequences in infinite-dimensional normed spaces is guaranteed by Dvoretzky's theorem in convex geometry (for more background and a modern proof see e.g.\ \cite{proofDvoretzky}), which we use in the following form:

\begin{theorem}[Dvoretzky (1961) \cite{dvoretzky}]$\;$\\
    For every $\varepsilon>0$ and $k \in \mathbb{N},$ there exists $N \in \mathbb{N}$ such that every normed space of dimension $N$ contains an  $\varepsilon$-almost-euclidean sequence of length $k.$
\end{theorem}

\begin{corollary} \label{Cor2}$\;$\\
    Let $(I_n)_n$ be an increasing interval-partition of $I \subseteq \mathbb{N}$ and let $\varepsilon>0.$ \\
    Every infinite-dimensional normed space contains a sequence that is $\varepsilon$-almost-euclidean on $(I_n)_n.$
\end{corollary}

\noindent
This leads to the necessity of rapidity for the {UBP$^{\text stat}$}- and UBP-properties. We now give the proof of Lemma~\ref{nec rapid}.

    \begin{proof}[Proof of Lemma \ref{nec rapid}]
    If $\Fil$ is not rapid$^+$, one can (using Lemma~\ref{alternatievekaraktrerisatierapid}) choose an \mbox{$\Fil$-sta}\-tion\-ary set $I$ and an increasing interval-partition $(I_n)_n$ of $I$ such that $\Fil$ is slow with respect to $(I_n)_n.$ By Corollary \ref{Cor2}, one can choose next a sequence $(\varphi_n)_n$ in $X'$ which is $\varepsilon$-almost-euclidean on $(I_n)_n$ (where one has fixed an arbitrary $\varepsilon\in \mathbb{R_{ >\text{0}}}$). By Lemma~\ref{lemma}, one then finds a pointwise \mbox{$\Fil$-bounded} sequence $(\psi_n)_n$ in $X'$ which is not bounded on any infinite subset of $I$. This contradicts the stationary uniform \mbox{$\Fil$-bound}\-ed\-ness principle for~$X.$\EndProof
\end{proof}

\section{Conclusions and additional remarks}

Theorem~\ref{Frechetsuf}, Lemma~\ref{Nec P} and Lemma~\ref{nec rapid} can now be combined to obtain the following two characterisation theorems.

\begin{theorem}$\;$ \label{mainx}\\
    The following statements are equivalent for a filter~$\Fil$ on $\mathbb{N}$.
    \begin{enumerate}
        \item There exists an infinite-dimensional Banach space~$X$ such that the uniform \mbox{$\Fil$-bound}\-ed\-ness principle holds for $X.$
        \item The uniform \mbox{$\Fil$-bound}\-ed\-ness principle holds for every Fr\'echet space~$X.$
        \item $\Fil$ is both a rapid$ ^+$- and a {\sterkePfilter}.
    \end{enumerate}
\end{theorem}

\begin{theorem}$\;$ \label{main+}\\
    The following statements are equivalent for a filter~$\Fil$ on $\mathbb{N}$.
    \begin{enumerate}
        \item There exists an infinite-dimensional Banach space~$X$ such that the stationary uniform \mbox{$\Fil$-bound}\-ed\-ness principle holds for $X.$
        \item The stationary uniform \mbox{$\Fil$-bound}\-ed\-ness principle holds for every Fr\'echet space~$X.$
        \item $\Fil$ is both a rapid$^+$- and {\zwakkePfilter}.
    \end{enumerate}
\end{theorem}

\noindent
In particular, the classes of ({stationary}) Fr\'echet-UBP-filters and ({stationary}) Banach-UBP-filters coincide,  we will now simplify terminology and refer to these filters as ({stationary}) UBP-filters.

\begin{example}[\cite{repicky}] $\,$\label{vb}\\
    Theorem~\ref{mainx} allows to give an elementary example of a UBP-filter~$\mathcal{G}$ which is not countably generated. This example was suggested to the authors by Miroslav Repick\'y.\\
    Let $\Nat = \bigsqcup_{i\in \Nat} E_i$ be a partition of $\Nat$ into finite sets $E_i$. Suppose in addition that $\limsup_{i\in \Nat} |E_i| = \infty.$ We consider the following filter:
    $$\mathcal{G}  := \{ A \subseteq \Nat : \limsup_{i\in \Nat} | E_i \setminus A| < \infty \}.$$
    Using that $A \subseteq \Nat$ is $\mathcal{G}$-stationary if and only if
    $$\limsup_{i\in \Nat} | E_i \cap A| = \infty,$$
     it is easily checked that $\mathcal{G}$ is a rapid$^+$ {\sterkePfilter} and hence, by Theorem~\ref{mainx}, also a UBP-filter. Moreover, given any sequence $(S_n)_n$ of elements of $\mathcal{G}$, it is possible to recursively construct increasing sequences of natural numbers
     $l_0 <l_1 < l_2 < l_3< \ldots$ and $k_0 <k_1 <k_2 < k_3 < \ldots$ such that $k_n \in E_{l_n} \cap S_n$. Then the set $\Nat \setminus \{ k_n : n \in \Nat \}$ belongs to $\mathcal{G}$ but does not contain any of the sets $S_n$ as subset. It follows that $\mathcal{G}$ can not be countably generated.
\end{example}

\subsection{UBP-ultrafilters}
As the filter~$\mathcal{G}$ in the previous example is only $F_\sigma$ in the Cantor space $P(\Nat),$ this example shows that UBP-filters of uncountable character need not be complex in the sense of descriptive set theory.
To study the other extreme, we now focus on the special case of UBP-ultrafilters. Adhering to our earlier convention, all ultrafilters we consider are supposed to be free.
For ultrafilters, the notions of {\zwakkePfilter} and {\sterkePfilter} coincide and one uses the term P-points to denote those ultrafilters which are also \mbox{{\sterkePfilter}s}.
The rapid$^+$ P-points are referred to as semi-selective ultrafilters.
Semi-selective ultrafilters are also characterised by the following combinatorial property (see e.g.~\cite{blass2}).\smallskip

\begin{lemma}$\;$\\
    $\mathcal{U}$ is a semi-selective ultrafilter if and only if for every sequence $(A_k)_k$ of elements of $\Ul$, there exists $\{a_0 < a_1 < \ldots\} \in \Ul$ with $a_k \in A_k$ for every~$k.$
\end{lemma}

\noindent
In the special case of ultrafilters, Theorem~\ref{mainx} and Theorem~\ref{main+} both reduce to:

\begin{corollary}$\,$\label{cor-ZFC}\\
    The set of UBP-ultrafilters coincides with the set of semi-selective ultrafilters. Moreover, if $\mathcal{U}$ is an ultrafilter and the uniform $\mathcal{U}$-boundedness principle holds for some infinite-dimensional Banach space~$X$, then $\mathcal{U}$ is semi-selective.
\end{corollary}

\noindent
This corollary extends a theorem by Benedikt (\cite[Proposition~3]{benedikt}) which expresses in the language of nonstandard analysis that every selective ultrafilter is a Banach-UBP-filter. An ultrafilter~$\Ul$ is selective if every $\Gamma:\Nat \to \Nat$ is either constant or injective on some $A\in \Ul.$ In particular, every selective ultrafilter is semi-selective.\\
Since the existence of selective ultrafilters follows from Martin's Axiom, the existence of UBP-ultrafilters is consistent with $\mathsf{ZFC}$  (provided the theory $\mathsf{ZFC}$ is itself consistent). In particular:
\begin{theorem}[\cite{halbeisen}]$\;$\\
If Martin's Axiom for countable posets (\,$\mathsf{MA}(countable)$) holds, then there exist \mbox{$2^{2^{\aleph_0}}$-many} non-isomorphic selective ultrafilters.
\end{theorem}

\noindent
It follows that $\mathsf{MA}(countable)$ also implies the existence of $2^{2^{\aleph_0}}$-many non-isomorphic UBP-ultrafilters. $\mathsf{MA}(countable)$ is a proper weakening of $\mathsf{MA}$ and is equivalent to the statement that there is no way of writing $\mathbb{R}$ as a union $\bigcup_{\alpha<\kappa} N_\alpha$ of $\kappa < 2^{\aleph_0}$ meager sets $N_\alpha \subseteq \mathbb{R}$ (see \cite{halbeisen}). \smallskip

\noindent
However, since the existence of P-points as well as the existence of rapid$^+$ ultrafilters is independent of the axioms of $\mathsf{ZFC}$-set theory, it follows that there are also no $\mathsf{ZFC}$-proofs of the existence of UBP-ultrafilters (again, of course, assuming consistency of $\mathsf{ZFC}$).
In particular, examples of models in which there do not exist UBP-ultrafilters include models obtained by adding $\aleph_2$ random reals to a model of $\mathsf{ZFC} + \mathsf{GCH}$~\cite{kunen}, Laver's model for the Borel conjecture \cite{miller} and Shelah's model for the absence of P-points \cite{shelah}.\smallskip

\subsection{Products and Quotients}
Theorem~\ref{mainx} and Theorem~\ref{main+} leave open the possibility that the uniform \mbox{$\Fil$-bound}\-ed\-ness principle could hold for some infinite-dimensional Fr\'echet space without $\Fil$ being a UBP-filter. We will now show that this can indeed occur. To this purpose, we first indicate how the validity of the uniform \mbox{$\Fil$-bound}\-ed\-ness principle behaves under countable products and quotients of Fr\'echet spaces.
\smallskip

\noindent
Let $\pi: X \to Y$ be a quotient map between locally convex spaces \cite[p.\ 33]{treves}. Then, every pointwise \mbox{$\Fil$-bounded} sequence $(T_i)_i$ of continuous linear maps from $Y$ to a locally convex space~$Z$ induces the sequence $(T_i\circ \pi )_i \in \mathcal{L}(X,Z),$ which is still pointwise \mbox{$\Fil$-bounded}. Moreover, $(T_i)_i$ is \mbox{$\Fil$-equi}\-con\-tin\-uous if and only if the sequence $(T_i\circ \pi )_i$ is \mbox{$\Fil$-equi}\-con\-tin\-uous.
This proves the following lemma.

\begin{lemma}$\;$ \label{quotienten}\\
    If $\pi: X \to Y$ is a quotient map between locally convex spaces and the uniform \mbox{$\Fil$-bound}\-ed\-ness principle holds for $X,$ then the uniform \mbox{$\Fil$-bound}\-ed\-ness principle holds for $Y$ as well.
\end{lemma}

\noindent
Likewise, we have:

\begin{lemma}$\;$ \label{quotienten+}\\
    If $\pi: X \to Y$ is a quotient map between locally convex spaces and the stationary uniform \mbox{$\Fil$-bound}\-ed\-ness principle holds for $X,$ then the stationary uniform \mbox{$\Fil$-bound}\-ed\-ness principle holds for $Y$ as well.
\end{lemma}

\noindent
For every Fr\'echet space~$X$ admitting an infinite-dimensional normable quotient, we can now characterise those filters $\Fil$ for which the (stationary) uniform \mbox{$\Fil$-bound}\-ed\-ness principle holds for $X.$

\begin{lemma}$\;$ \label{normeerbaar quotient}\\
    Let $X$ be a Fr\'echet space with an infinite-dimensional normable quotient.
    \begin{itemize}
        \item The uniform \mbox{$\Fil$-bound}\-ed\-ness principle holds for $X$ if and only if $\Fil$ is a UBP-filter.
        \item The stationary uniform \mbox{$\Fil$-bound}\-ed\-ness principle holds for $X$ if and only if $\Fil$ is a {stationary} UBP-filter.
    \end{itemize}
\end{lemma}
\begin{proof}
    The other direction being evident from the definition of ({stationary}) UBP-filter, it suffices to prove that $\Fil$ is a ({stationary}) UBP-filter whenever the (stationary) uniform \mbox{$\Fil$-bound}\-ed\-ness principle holds for $X$.\\ Suppose that the (stationary) uniform \mbox{$\Fil$-bound}\-ed\-ness principle holds for $X$. Then $\Fil$ is a (weak) {\sterkePfilter}, by Lemmas~\ref{P  voor niet normeerbaar} and \ref{Nec P}. Because of Lemma~\ref{quotienten} and Lemma~\ref{quotienten+}, validity of the (stationary) uniform \mbox{$\Fil$-bound}\-ed\-ness principle will be inherited by the infinite-dimensional normable quotient of $X,$ so that it follows from Lemma~\ref{nec rapid} that $\Fil$ is also rapid$^+$. The conclusion follows from Theorem~\ref{Frechetsuf}.
\EndProof
\end{proof}

\noindent        
As long as $\Fil$ is a (weak) {\sterkePfilter}, the (stationary) uniform \mbox{$\Fil$-bound}\-ed\-ness principle is also preserved under countable products of Fr\'echet spaces.

\begin{lemma}$\;$ \label{producten}\\
    Let $\Fil$ be a {\sterkePfilter}.
    If $(X_i)_{i\in \Nat}$ is a sequence of (possibly finite-dimensional) Fr\'echet spaces and
    $$X = \prod_{i\in \Nat} X_i,$$
    then the uniform \mbox{$\Fil$-bound}\-ed\-ness principle holds for $X$ if and only if  the uniform \mbox{$\Fil$-bound}\-ed\-ness principle holds for every one of the spaces $X_i.$
\end{lemma}
\begin{proof}
    Let $Y$ be an arbitrary locally convex space and let $Q_i$ be a base of continuous seminorms for $X_i$.\\
    \framebox{$\Rightarrow$}
    Since the projection on the $i$-th coordinate is a quotient map from $X$ to $X_i$, this direction follows from Lemma~\ref{quotienten}.\\
    \framebox{$\Leftarrow$}
    For the other direction, let $(T_j)_j$ be a pointwise \mbox{$\Fil$-bounded} sequence in $\mathcal{L}(X,Y)$. In order to show that $(T_j)_j$ is \mbox{$\Fil$-equi}\-con\-tin\-uous, it suffices to show that for an arbitrary continuous seminorm $r$ on $Y$, $(T_j)_j$ is \mbox{$\Fil$-equi}\-con\-tin\-uous in $\mathcal{L}(X, (Y,r))$. Note that, making use of the quotient mapping from $(Y,r)$ onto the normed space $(Y / \ker(r), r)$, we can assume that $r$ is a norm. We will thus write $Y$ for the normed space $(Y,r)$ in the rest of this proof. Observe that then
    $$ \mathcal{L}(X,Y) = \mathcal{L}(\prod_{i\in \Nat} X_i,Y) = \bigoplus _{i\in \Nat} \mathcal{L}(X_i,Y),$$
    in particular, for every $j$ we find $(P^j_i)_{0\leq i \leq n_j},$ with $P^j_i \in \mathcal{L}(X_i,Y)$,  such that
    \begin{equation}T_j( (x_n)_n ) = \sum_{0\leq i \leq n_j} P^j_i(x_i), \qquad \text{for all } (x_n)_n \in X.\tag{$\ast$}\label{eq:1}\end{equation}
    Since $\Fil$ is a {\sterkePfilter} and $X$ is a Fr\'echet space, it follows from Lemma~\ref{puntsgewijsbegrensddanookopomgeving} that there exists a \mbox{$0$-neigh}\-bour\-hood $B$ in $X$ and $F\in \Fil$ such that
    $$\sup_{x\in B} r(T_j(x)) < \infty, \qquad \text{for every } j \in F.$$
    Because $B$ is a \mbox{$0$-neigh}\-bour\-hood in the product space~$X$, there exist $N \in \Nat$, $c > 0$ and $q_0, \ldots,q_N$ with $q_i \in Q_i$ such that
    $$\bigcap_{i \leq N} \{ (x_n)_n \in X : q_i(x_i)\leq c \} \subseteq B.$$  Since every mapping $r \circ T_j,$ with $j\in F,$ is also bounded on this $B,$ one finds that $r \circ P^j_i$ is necessarily identically zero for every  $i > N$ and every $j \in F.$ It is therefore sufficient to prove that the sequence $(P^j_i)_j$ is \mbox{$\Fil$-equi}\-con\-tin\-uous for every $i \leq N$. It follows from~(\ref{eq:1}) that the sequence $(P^j_i)_j$ is pointwise \mbox{$\Fil$-bounded} for every $i$, so we can conclude \mbox{$\Fil$-equi}\-con\-tin\-uity of this sequence by applying the uniform \mbox{$\Fil$-bound}\-ed\-ness principle for the space~$X_i.$
\EndProof
\end{proof}

\noindent
With an analogous argument, one proves the corresponding lemma for the stationary uniform \mbox{$\Fil$-bound}\-ed\-ness principle.

\begin{lemma}$\;$ \label{producten+}\\
    Let $\Fil$ be a {\zwakkePfilter}.
    If $(X_i)_{i\in \Nat}$ is a sequence of (possibly finite-dimensional) Fr\'echet spaces and
    $$X = \prod_{i\in \Nat} X_i,$$
    then the stationary uniform \mbox{$\Fil$-bound}\-ed\-ness principle holds for $X$ if and only if the stationary uniform \mbox{$\Fil$-bound}\-ed\-ness principle holds for every one of the spaces $X_i.$
\end{lemma}

\bigskip

\noindent
As a corollary, we find an example of an infinite-dimensional Fr\'echet space for which the (stationary) uniform \mbox{$\Fil$-bound}\-ed\-ness principle can hold for filters $\Fil$ which are not ({stationary}) UBP-filters.

\begin{corollary}\label{RN} $\;$\\
    For the sequence space $\mathbb{R}^\Nat$, equipped with the product topology,
    \begin{itemize}
        \item the uniform \mbox{$\Fil$-bound}\-ed\-ness principle holds if and only if $\Fil$ is a {\sterkePfilter},
        \item the stationary uniform \mbox{$\Fil$-bound}\-ed\-ness principle holds if and only if $\Fil$ is a {\zwakkePfilter}.
    \end{itemize}
    
\end{corollary}
\begin{proof}
Because $\mathbb{R}^\Nat$ is not normable, the {\it only if} directions follow from Lemma~\ref{P  voor niet normeerbaar}.\\
But the (stationary) uniform \mbox{$\Fil$-bound}\-ed\-ness principle holds trivially for the space~$\mathbb{R},$ so it follows from the previous theorem that the (stationary) uniform \mbox{$\Fil$-bound}\-ed\-ness principle will also hold for $\mathbb{R}^\Nat,$ whenever $\Fil$ is a (weak) {\sterkePfilter}.
\EndProof
\end{proof}

\noindent
We have thus obtained a characterisation of the (weak) {\sterkePfilter s} in terms of the  topological vector space structure on $\mathbb{R}^\Nat.$\bigskip

\noindent
The following Remark~\ref{remKothe} and Example~\ref{exCm} each exhibit an illustration of Lemma~\ref{normeerbaar quotient}.

\begin{remark} \label{remKothe}
Since $\mathbb{R}^\mathbb{N}$ is a Fr\'echet-Montel space, Corollary~\ref{RN} prompts the question whether every Fr\'echet-Montel space necessarily satisfies the uniform \mbox{$\Fil$-bound}\-ed\-ness principle for every {\sterkePfilter}  $\Fil$. It follows from Lemma~\ref{normeerbaar quotient} that this is not the case as there exist Fr\'echet-Montel spaces admitting infinite dimensional Banach quotients. (See \cite[11.6.4]{jarchow} for an example of a K\"othe sequence space that is Fr\'echet-Montel but surjects continuously to $\ell_1$.)
\end{remark}

\begin{example}$\;$ \label{exCm}\\
    By Lemma~\ref{normeerbaar quotient}, the spaces of $m$-times continuously differentiable functions\linebreak $C^m(\Omega)$,
    with $\Omega \subseteq \mathbb{R}^n$ open and $m\in \mathbb{N}$ satisfy the (stationary) uniform \mbox{$\Fil$-bound}\-ed\-ness principle if and only if $\Fil$ is a
    (stationary) UBP-filter. Indeed, the \mbox{Banach} spaces $\mathcal{E}^m(K)$ of Whitney differentiable functions on compact convex
    $K \subseteq \Omega$ appear as quotients of the spaces $C^m(\Omega)$
     (cf.\ \cite[Theorem 2.2]{frerick}).
\end{example}
\bigskip

\noindent
In contrast to the spaces $C^m(\Omega)$ ($m\in\mathbb{N}$) in the previous example, the space $C^\infty(\Omega)$ belongs to the class of Fr\'echet-Schwartz spaces (just as the space~$\mathbb{R}^\Nat$). As every Hausdorff quotient of a Schwartz space is again Schwartz and the only normable Schwartz spaces are finite-dimensional, no infinite-dimensional quotient of a Schwartz space can be normable.
By consequence, if $X$ is Schwartz, Lemma~\ref{normeerbaar quotient} does not carry any extra information about the validity of the uniform \mbox{$\Fil$-bound}\-ed\-ness principle for $X.$\\ The following problem therefore remains open.

\begin{question}
    For which Fr\'echet-Schwartz spaces $X$ and which filters $\Fil$ does the (stationary) uniform \mbox{$\Fil$-bound}\-ed\-ness principle hold for $X$?
\end{question}

\noindent
This question is a special case of the following more general question asking whether ({stationary}) UBP-filters are necessary for uniform \mbox{$\Fil$-bound}\-ed\-ness principles in those cases left open by Lemma~\ref{normeerbaar quotient}.

\begin{question}
    If $X$ is a Fr\'echet space that does not admit an infinite-dimensional Banach quotient, for which filters $\Fil$ does the (stationary) uniform \mbox{$\Fil$-bounded} principle hold for~$X$?
\end{question}

\noindent
Because of Lemma~\ref{P  voor niet normeerbaar}, every such filter~$\Fil$ should be at least a (weak) {\sterkePfilter} and the question comes down to determining exactly for which spaces the extra rapidity$^+$ assumption is (fully) necessary.\\
\\
Finally, we point to the study of uniform \mbox{$\Fil$-bound}\-ed\-ness principles in the more \mbox{general} setting of (not necessarily Fr\'echet) barrelled topological vector spaces, which was not addressed here, but could be of interest for further research.

\nocite{*}
\bibliographystyle{plainnat}
\bibliography{bibliografieUBP}

$\;$
\bigskip

\noindent
Department of Mathematics: Analysis, Logic and Discrete Mathematics,\\ Ghent University, Krijgslaan 281, B-9000 Ghent, Belgium
\smallskip

\noindent
E-mail addresses:\\ ben.de-bondt@imj-prg.fr (Ben De Bondt)\\
hans.vernaeve@ugent.be (Hans Vernaeve)

\end{document}